\documentclass[11pt]{article}
\usepackage{amsmath,fullpage,amssymb,amsthm,hyperref,enumerate}
\usepackage{tikz,color,cases}
\usetikzlibrary{patterns}
\usepackage{marvosym,ifthen}

\newtheorem{theorem}{Theorem}[section]

\newtheorem{lemma}[theorem]{Lemma}

\theoremstyle{definition}
  
\newtheorem{definition}[theorem]{Definition}   

\theoremstyle{remark}

\renewcommand{\P}{\mathcal P}
\renewcommand{\S}{\mathfrak S}
\newcommand{\G}{\mathcal G}

\newcommand{\Z}{\mathbb Z}

\DeclareMathOperator\maj{maj}

\newcommand{\qbin}[2]{\begin{bmatrix}{#1}\\ {#2}\end{bmatrix}_q}
\newcommand{\qbinsq}[2]{\begin{bmatrix}{#1}\\ {#2}\end{bmatrix}_{q^2}}
\newcommand{\bij}{\theta}
\newcommand{\Bij}{\Theta}

\DeclareMathOperator\cro{\chi}
\newcommand\start{circle(.07)}
\newcommand\startud{circle(.07)}
\newcommand\E{-- ++(1,0) \start}
\newcommand\N{-- ++(0,1) \start}
\newcommand\dn{-- ++(1,-1) \startud}
\newcommand\up{-- ++(1,1) \startud}
\newcommand{\si}{\sigma}
\newcommand{\ta}{\tau}
\newcommand{\bs}{\bar\sigma}
\newcommand{\bt}{\bar\tau}
\newcommand{\om}{\omega}
\newcommand{\bom}{\overline\omega}

\newcommand{\de}{\circ}
\newcommand{\bde}{\bullet} 
\newcommand{\vv}{\mathbf v}
\newcommand{\uu}{\mathbf u}
\newcommand{\nn}{\mathbf n}
\newcommand{\ee}{\mathbf e}
\newcommand{\bP}{\mathbf P}
\newcommand{\fP}{\mathfrak P}
\DeclareMathOperator\sgn{sgn}

\newcommand{\phiC}{\phi_C}
\newcommand{\rt}{\vdash}
\newcommand{\lt}{\dashv}

\newenvironment{list1}{
  \begin{list}{}{%
      \setlength{\itemsep}{2mm}
      \setlength{\parsep}{1mm} \setlength{\parskip}{1mm}
      \setlength{\topsep}{0mm} \setlength{\partopsep}{0in}
      \setlength{\leftmargin}{0in}}}{\end{list}}

\newcommand\paths[4]{\P_{#1\to #2,#3\to #4}}
\newcommand\pathsP[5]{\P_{#1\to #2,#3\to #4}^{\ge #5}}
\newcommand\pathsN[5]{{\mathcal N}_{#1\to #2,#3\to #4}^{\ge #5}}
\newcommand\pathsE[5]{{\mathcal E}_{#1\to #2,#3\to #4}^{\ge #5}}
\newcommand\pathsNC[4]{{\mathcal N}_{#1\to #2,#3\to #4}^{C}}
\newcommand\pathsPD[4]{{\mathcal P}_{#1\to #2,#3\to #4}^{\Join}}

\newcommand\markx[4]{
\draw (#1,.2)--(#1,-.2) node[below] {$#3$};
\draw (.2,#2)--(-.2,#2) node[left] {$#4$};
\draw[dotted] (#1,.3)--(#1,#2)--(.3,#2);
}
\newcommand\marky[4]{\markx{#2}{#1}{#4}{#3}}

\newcommand\crossing[2]{\draw[thick] (#1,#2) circle (.2);}
\newcommand\intersection[2]{\draw[thick,fill] (#1,#2) circle (.2);}
\newcommand\cutting[2]{\draw[thick] (#1-.3,#2+.3)--(#1+.3,#2-.3); \draw[thick] (#1-.3,#2-.3)--(#1+.3,#2+.3);}
\newcommand\cuttingthin[2]{\draw[thick,lightgray] (#1-.3,#2+.3)--(#1+.3,#2-.3); \draw[thick,lightgray] (#1-.3,#2-.3)--(#1+.3,#2+.3);}

\newcommand\diamant[2]{\draw[teal] (#1-.25,#2)--(#1,#2+.25)--(#1+.25,#2)--(#1,#2-.25)--(#1-.25,#2);}

\newcommand\movepeakvalley[2]{
\diamant{#1}{#2} 
\diamant{#1}{-#2} 
\ifthenelse{#2=0}{}{\draw[teal,dotted,thick,->,shorten <=4pt,shorten >=4pt] (#1,#2)--(#1,-#2);}
}

\author{Sergi Elizalde\thanks{Department of Mathematics, Dartmouth College, Hanover, NH 03755, USA. {\tt sergi.elizalde@dartmouth.edu}.}}

\title{Counting lattice paths by crossings and major index I:\\ the corner-flipping bijections}

\date{}

\begin{document}

\maketitle

\begin{abstract}
We solve two problems regarding the enumeration of lattice paths in $\Z^2$ with steps $(1,1)$ and $(1,-1)$ with respect to the major index, defined as the sum of the positions of the valleys, and to the number of certain crossings.
The first problem considers crossings of a single path with a fixed horizontal line.
The second one counts pairs of paths with respect to the number of times they cross each other. 
Our proofs introduce lattice path bijections with convenient visual descriptions, and the answers are given by remarkably simple formulas involving $q$-binomial coefficients. 
\end{abstract}

\section{Introduction}\label{sec:intro}

\subsection{Background}

The enumeration of lattice paths is an important topic both in combinatorics and in mathematical statistics, as discussed in the surveys by Mohanty~\cite{Moh} and Krattenthaler~\cite{Krat}.
In the particular case of lattice paths in the plane with two types of steps, common questions involve counting paths constrained by some boundary, as well as counting paths with respect to various statistics.

One important such statistic is the number of times that a path crosses a given line. Several instances of the enumeration of paths by this kind of statistic have appeared in the probability and statistics literature~\cite{Eng,Sen,Feller-book,KW,Spivey}, often resulting in nice formulas involving binomial coefficients. Another statistic commonly studied in the combinatorics literature is the {\em major index} of a path, which can be defined as the sum of the positions of its turns in a given direction. This statistic on paths, which arises naturally when interpreting them as a binary words, has been studied, for example, in~\cite{KM, Krat-turns,Krat-nonint,SaSa}.

In this paper we consider the enumeration of lattice paths simultaneously by the number of crossings and the major index. We show that, rather surprisingly, the resulting enumeration formulas with respect to both statistics are quite simple, having closed forms in terms of $q$-binomial coefficients.
Intriguingly, the methods that have been commonly used to count paths by the number of crossings do not give an obvious explanation for such simple formulas. However, in all cases, we are able to prove them bijectively.

We consider two different but related problems. The first one concerns single lattice paths, which will be enumerated with respect to the major index and the number of times that they cross a fixed line. The second one involves pairs of lattice paths, which will be enumerated with respect to the sum of their major indices and the number of times that they cross each other. 
The tools used to solve both problems are similar, and they involve certain lattice path bijections that, unlike classical methods such as the reflection principle and prefix-swapping operations, behave well with respect to the major index.

In the case of zero crossings, our work relates to the important topic of non-crossing (or non-intersecting, after a simple transformation) paths, which have been studied for decades.
The celebrated determinantal formula by Gessel and Viennot \cite{GV} enumerating tuples of non-intersecting paths, previously discovered by Lindstr\"om~\cite{Lin} in the context of matroid theory, has connections to symmetric functions, tableaux, plane partitions and tilings, and even to statistical physics \cite{Fisher} and chemistry. A refinement of this formula that keeps track of the sum of the major indices of the paths has been given by
Krattenthaler~\cite{Krat-nonint}. 
Krattenthaler's formula in the special case of two paths is equivalent to our formula for pairs of paths in the special case of zero crossings. In Section~\ref{sec:non-int}, we will show how our tools also yield an alternative proof of Krattenthaler's formula. While the ideas behind both proofs are similar, our bijections have simple descriptions directly in terms of paths, whereas the bijections in~\cite{Krat-nonint} require passing through other objects called two-rowed arrays.
We point out that it is an open question whether our formulas that enumerate pairs of paths with a given number of crossings can be extended to $k$-tuples of paths for $k>2$.

In the special case of pairs of paths with at least one common endpoint, there has been work by Gessel et al.~\cite{GGSWY} enumerating such pairs with respect to the number of lattice points where the paths intersect. 
There is, however, no direct relationship between this statistic and the number of crossings that we consider here, so it is no surprise that the 
summation formulas obtained in~\cite{GGSWY} are different from ours. 

It is important to note that, even though this paper focuses exclusively on 
lattice paths, our work has applications to the enumeration of integer partitions with constrained ranks. 
Specifically, our formula for paths crossing a line is one of the tools that is used in a forthcoming paper by Corteel et al.~\cite{CES} to enumerate partitions with a given number of off-diagonal rank parity blocks, which generalizes results of Seo and Yee~\cite{SeoYee}.

Finally, another follow-up paper~\cite{part2} will further refine our results by another statistic: the {\em number of descents}, which also arises naturally when interpreting paths as binary words, and can be described as the number of turns of the path in a given direction. The proofs of the refined version can no longer be visualized as lattice path bijections,
but rather they are based on certain two-rowed arrays that have been previously used by Krattenthaler and Mohanty~\cite{Krat-turns,Krat-nonint,KM}.

\subsection{Basic definitions}\label{sec:basic}

We consider simple  
lattice paths in $\Z^2$ with steps $U=(1,1)$ and $D=(1,-1)$ (up and down), although sometimes it will be convenient to consider the steps to be $N=(0,1)$ and $E=(1,0)$ (north and east) instead. One type of paths is obtained from the other by rotating $45^\circ$ and stretching by a factor of $\sqrt{2}$; equivalently, using the substitutions $U\leftrightarrow N$ and $D\leftrightarrow E$. We will use both settings interchangeably.

For nonnegative integers $a,b$, let $\G_{a,b}$ denote the set of paths with $a$ steps $U$ and $b$ steps $D$, usually starting at the origin, although later it will be convenient to allow other initial points on the $y$-axis. 
The sequence of steps of such a path can be encoded as a binary word with $a$ zeros and $b$ ones, by identifying $U$s with $0$s and $D$s with $1$s\footnote{We follow the convention of F\"urlinger and Hofbauer~\cite{FH}. Other papers use a different encoding whereby descents of the word become {\em peaks} of the path, and $\maj(P)$ is defined be the sum of the $x$-coordinates of the peaks of $P$.}. Under this encoding, descents of the word correspond to {\em valleys} of the path, defined as vertices that are preceded by a $D$ and followed by a $U$. 
The major index, which is a common statistic on words, can then be translated to paths $P\in\G_{a,b}$, by defining $\maj(P)$ to be the sum of $x$-coordinates of the valleys of $P$. See Figure~\ref{fig:Gstrl} for an example.

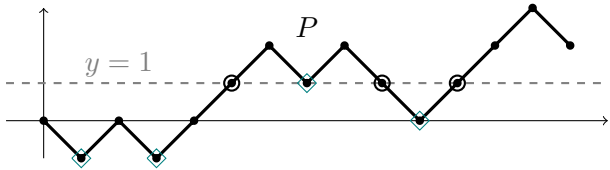
\begin{figure}[htb]
\centering
\begin{tikzpicture}[scale=0.5]
\draw[->] (-1,0)--(15,0);
\draw[->] (0,-1)--(0,3);
\draw[dashed,thick,gray] (-1,1)--(15,1);
\draw (2,1.5) node[gray] {$y=1$};
\draw[very thick,fill] (0,0) \startud \dn\up\dn\up\up\up\dn\up\dn\dn\up\up\up\dn;
\draw (7,2.5) node {$P$};
\crossing{5}{1}
\crossing{9}{1}
\crossing{11}{1}
\diamant{1}{-1}
\diamant{3}{-1}
\diamant{7}{1}
\diamant{10}{0}
\end{tikzpicture}
\caption{A path $P\in\G_{8,6}^{\ge 3,1}$ with $\maj(P)=1+3+7+10=21$. The four valleys are marked with teal diamonds, and the
three crossings of the line $y=1$ are circled in black. The middle crossing is a downward crossing, whereas the other two are upward crossings.}
\label{fig:Gstrl}
\end{figure}

When dealing with paths with $N$ and $E$ steps, we denote by $\P_{A\to B}$ the set of such paths that start at $A$ and end at $B$, where $A,B\in\Z^2$.  In this case, the {\em valleys} of $P\in\P_{A\to B}$ are the vertices preceded by an $E$ and followed by an $N$, and $\maj(P)$ is the sum of the positions of the valleys, where the position is determined by numbering the vertices of the path in increasing order starting at $A$, which would be position $0$, and ending at $B$. Note that $P_{(x,y)\to(u,v)}$ is empty unless $x\le u$ and $y\le v$.

Recall the {\em $q$-binomial coefficients}, defined as
$$\qbin{n}{k}=\frac{(1-q^n)(1-q^{n-1})\cdots(1-q^{n-k+1})}{(1-q^k)(1-q^{k-1})\cdots(1-q)}
$$
if $0\le k\le n$, and as $0$ otherwise.
The following is a classical result of MacMahon.

\begin{lemma}[\cite{Mac}] \label{lem:qbin}
For $a,b\ge0$,
$$\sum_{P\in\G_{a,b}} q^{\maj(P)} = \qbin{a+b}{a}.$$
Equivalently, if $A=(x,y)$ and $B=(u,v)$, then
$$\sum_{P\in\P_{A\to B}} q^{\maj(P)}=\qbin{u-x+v-y}{u-x}.$$
\end{lemma}

The rest of the paper is structured as follows. In Section~\ref{sec:main} we state our main theorems, after establishing some more definitions and notation. Section~\ref{sec:ingredients} introduces some tools, in particular four closely related bijections, that will play a key role in our proofs. Section~\ref{sec:line-proofs} applies these tools to prove our results from Section~\ref{sec:line} about the enumeration of paths by the number of crossings of a horizontal line, while Section~\ref{sec:pairs-proofs} applies them to prove our results from Section~\ref{sec:pairs} about the enumeration of pairs of paths by the number of times they cross each other. Our proofs in the two settings have certain similarities, but Section~\ref{sec:pairs-proofs} can be read independently from Section~\ref{sec:line-proofs}. 
Section~\ref{sec:connections} uses our construction to give an alternative proof of Krattenthaler's refined enumeration of tuples of nonintersecting paths~\cite{Krat-nonint}.
We discuss possible extensions of our work in Section~\ref{sec:further}.

\section{Main results}\label{sec:main}

\subsection{Paths crossing a line}\label{sec:line}

First we state our results about the enumeration of paths with $U$ and $D$ steps with respect to the major index and to the number of times that they cross a horizontal line. For $\ell,r\in\Z$, where $r\ge0$, let $\G_{a,b}^{\ge r,\ell}$ denote the set of paths in $\G_{a,b}$ that cross the line $y=\ell$ at least $r$ times.
For this definition, a vertex of the path on the line $y=\ell$ is a {\em crossing} if it is either preceded and followed by a $D$ (in which case it is called a {\em downward crossing}), or preceded and followed by a $U$ (called an {\em upward crossing}). See Figure~\ref{fig:Gstrl} for an example.

We are interested in the polynomials
$$G_{a,b}^{\ge r,\ell}(q)=\sum_{P\in\G_{a,b}^{\ge r,\ell}} q^{\maj(P)}.$$
The polynomials that count paths crossing the line $y=\ell$ {\em exactly} $r$ times can be easily expressed in terms of these as $$G_{a,b}^{=r,\ell}(q)=G_{a,b}^{\ge r,\ell}(q)-G_{a,b}^{\ge r+1,\ell}(q).$$
We will provide a formula for $G_{a,b}^{\ge r,\ell}(q)$ for arbitrary $a,b,r,\ell\in\Z$ with $a,b,r\ge0$. The formula is slightly different depending on whether the starting and ending points of the path are above, below, or on the line being crossed. In each case, the resulting expressions are surprisingly
simple, consisting of a $q$-binomial coefficient times a power of $q$.
Despite the simple formulas, our proof is by no means trivial. 
In each case, we provide a bijection from $\G_{a,b}^{\ge r,\ell}$ to a set of paths with no requirements on the number of crossings, which can then be enumerated using Lemma~\ref{lem:qbin}. A key property of our bijection is that it has a predictable effect on the major index of the paths.

Let us first state the result in the case $\ell=0$, that is, when considering crossings of the $x$-axis.

\begin{theorem}\label{thm:xaxis}
For any $a,b,r\ge0$,
\begin{subnumcases}{G_{a,b}^{\ge r,0}(q)=} 
q^{\binom{r+1}{2}}\qbin{a+b}{a+r} & if $a>b$, \label{eq:0=l<a-b}\\
(1+q^a)q^{\binom{r+1}{2}}\qbin{2a-1}{a+r} & if $a=b$, \label{eq:0=l=a-b}\\
q^{\binom{r}{2}}\qbin{a+b}{a-r} & if $a<b$. \label{eq:0=l>a-b}
\end{subnumcases}
\end{theorem}

In this $\ell=0$ case, the specialization $q=1$ (i.e., when we disregard the major index) has been studied in the probability literature. The formula for $q=1$ has been known for over 50 years: it first appeared in work of Engelberg~\cite{Eng} and Sen~\cite{Sen}, and was later rediscovered by other authors~\cite{KM}. It refines a classical result of Feller~\cite{Feller,Feller-book} for paths without a fixed endpoint. Some of these papers also determine the limiting distribution of the number of crossings. The proofs in~\cite{Eng,Sen} consist essentially of repeatedly applying Andr\'e's reflection principle at each crossing.
Unfortunately, this method does not provide a proof of our refinement with the variable $q$, because the major index does not behave well under reflection of a piece of the path. Thus, proving Theorem~\ref{thm:xaxis} requires more sophisticated bijections that keep track of the statistic $\maj$.

The case $a>b$ of Theorem~\ref{thm:xaxis} can be shown to be equivalent to a result of Seo and Yee~\cite[Lemma 2.1]{SeoYee} concering the enumeration of ballot paths with marked returns, with respect to a different statistic that combines valleys and returns. Seo and Yee's proof is recursive, by induction on the length of the path, and so it does not give much insight on why the resulting formula is so simple. Similar ideas could be used to provide a recursive proof of Theorem~\ref{thm:xaxis}, but we prefer to present a bijective proof instead (see Section~\ref{sec:line-proofs}).

Next we state the result in the case $\ell\neq0$. The parity of $r$ plays a role in this case, so we write $r=2m$ or $r=2m\pm1$ for convenience.
Note that the results are trivial for $r=0$, since $\G_{a,b}^{\ge 0,\ell}=\G_{a,b}$ for any $\ell$, and so $G_{a,b}^{\ge 0,\ell}(q)$ is already given by Lemma~\ref{lem:qbin}.

\begin{theorem}\label{thm:line}
Let $a,b,m\ge0$, and let $\ell\in\mathbb{Z}\setminus\{0\}$.
If $0<\ell<a-b$, then
\begin{equation}\label{eq:0<l<a-b}
G_{a,b}^{\ge 2m+1,\ell}(q)=G_{a,b}^{\ge 2m,\ell}(q)=q^{m(2m+1+\ell)}\qbin{a+b}{a+2m}.
\end{equation}
If $0>\ell>a-b$, then
\begin{equation}\label{eq:0>l>a-b}
G_{a,b}^{\ge 2m+1,\ell}(q)=G_{a,b}^{\ge 2m,\ell}(q)=q^{m(2m-1-\ell)}\qbin{a+b}{a-2m}.
\end{equation} 
If $0>\ell<a-b$ and $m\ge1$, then
\begin{equation}\label{eq:0>l<a-b}
G_{a,b}^{\ge 2m,\ell}(q)=G_{a,b}^{\ge 2m-1,\ell}(q)=q^{m(2m-1-\ell)}\qbin{a+b}{a+2m-1-\ell}.
\end{equation}
If $0<\ell>a-b$ and $m\ge1$, then
\begin{equation}\label{eq:0<l>a-b}
G_{a,b}^{\ge 2m,\ell}(q)=G_{a,b}^{\ge 2m-1,\ell}(q)=q^{(m-1)(2m-1+\ell)}\qbin{a+b}{a-2m+1-\ell}.
\end{equation}
If $0<\ell=a-b$, then
\begin{equation}\label{eq:0<l=a-b}
G_{a,b}^{\ge 2m,\ell}(q)=q^{m(2m+1+\ell)}\qbin{a+b}{a+2m}, \quad G_{a,b}^{\ge 2m+1,\ell}(q)=q^{m(2m+1+\ell)}\qbin{a+b}{a+2m+1}.
\end{equation}
If $0>\ell=a-b$, then
\begin{equation}\label{eq:0>l=a-b}
G_{a,b}^{\ge 2m,\ell}(q)=q^{m(2m-1-\ell)}\qbin{a+b}{a-2m}, \quad G_{a,b}^{\ge 2m+1,\ell}(q)=q^{(m+1)(2m+1-\ell)}\qbin{a+b}{a-2m-1}.
\end{equation} 
\end{theorem}

We remark that the set $\G_{a,b}^{\ge r,\ell}$ is in trivial bijection with $\G_{b,a}^{\ge r,-\ell}$ (by reflecting the paths along the $x$-axis),
with $\G_{b,a}^{\ge r,\ell-a+b}$ (by reflecting the paths along a vertical line and translating appropriately), and with $\G_{a,b}^{\ge r,-\ell+a-b}$ (by composing both reflections, which is equivalent to rotating the paths by $180^\circ$). However, none of these  bijections changes $\maj$ in a consistent way unless the number of valleys or the last step of the path are fixed. Thus, the different cases in Theorems~\ref{thm:xaxis}  and~\ref{thm:line} cannot be trivially derived from each other even when the sets of paths are related by these reflections. Similarly, there is no obvious way to deduce Theorem~\ref{thm:line} from Theorem~\ref{thm:xaxis} by dettaching the portion of the path before the first crossing of the line $y=\ell$, since the removal of this prefix affects the major index inconsitently.

We will give a bijective proof of Theorems~\ref{thm:xaxis} and~\ref{thm:line} in Section~\ref{sec:line-proofs}, using some ingredients that we introduce in Section~\ref{sec:ingredients}.

\subsection{Pairs of paths crossing each other}\label{sec:pairs}

Next we enumerate pairs of paths according to the sum of their major indices and to the number of times that they cross each other. For convenience, we will consider paths with $N$ and $E$ steps for this problem.
A {\em crossing} of two paths $P$ and $Q$ is defined to be a common vertex $C$ such that 
\begin{itemize}
\item $P$ and $Q$ disagree in the step arriving at $C$, and they disagree again in some step after $C$;
\item at the first step after $C$ where $P$ and $Q$ disagree again, each path has the same type of step ($N$ or $E$) as it had when arriving at $C$.
\end{itemize}
See Figure~\ref{fig:crossing} for some examples.
Note that two paths can intersect (that is, have common vertices) without crossing. Let $\cro(P,Q)$ denote the number of crossings of paths $P$ and $Q$; see Figure~\ref{fig:pair} for an example. For $A_1,A_2,B_1,B_2\in\Z^2$ and $r\ge0$, we use the following notation for pairs of paths having at least $r$ crossings, where $\{\de,\bde\}=\{1,2\}$:
$$\pathsP{A_1}{B_\de}{A_2}{B_\bde}{r}=\{(P,Q):P\in\P_{A_1\to B_\de},Q\in\P_{A_2\to B_\bde},\cro(P,Q)\ge r\}.$$
Note that $\pathsP{A_1}{B_\de}{A_2}{B_\bde}{0}=\P_{A_1\to B_\de}\times\P_{A_2\to B_\bde}$; we denote this set simply by $\paths{A_1}{B_\de}{A_2}{B_\bde}$.

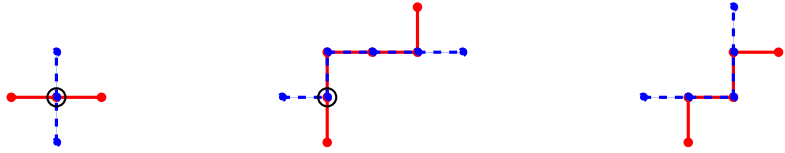
\begin{figure}[htb]
\centering
\begin{tikzpicture}[scale=0.6]
    \draw[red,very thick,fill](0,1) \start\E\E;
    \draw[blue,dashed,very thick,fill](1,0) \start\N\N;
\crossing{1}{1}
    \draw[red,very thick,fill](7,0) \start\N\N\E\E\N;
    \draw[blue,dashed,very thick,fill](6,1)\start\E\N\E\E\E;
\crossing{7}{1}    
    \draw[red,very thick,fill](15,0) \start\N\E\N\E;
    \draw[blue,dashed,very thick,fill](14,1)\start\E\E\N\N;
\end{tikzpicture}
\caption{Two examples of crossings, circled in black, and a pair of paths that do not cross (right).}
\label{fig:crossing}
\end{figure}

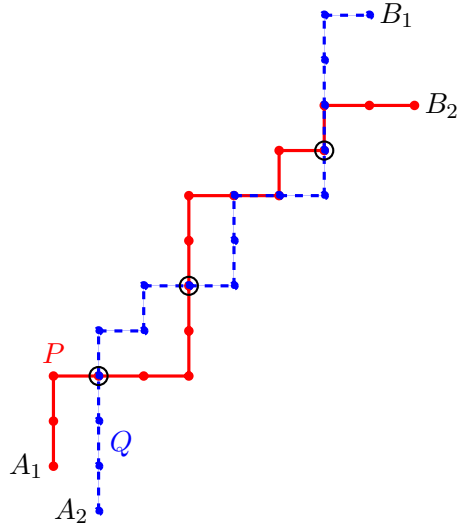
\begin{figure}[htb]
\centering
\begin{tikzpicture}[scale=0.6]
    \draw[red,very thick,fill](0,1) \start\N\N\E\E\E\N\N\N\N\E\E\N\E\N\E\E;
    \draw[blue,dashed,very thick,fill](1,0) \start\N\N\N\N\E\N\E\E\N\N\E\E\N\N\N\N\E;
\crossing{1}{3};\crossing{3}{5};\crossing{6}{8};
	\draw (0,1) node[left] {$A_1$};
	\draw (1,0) node[left] {$A_2$};     
	\draw (7,11) node[right] {$B_1$};
	\draw (8,9) node[right] {$B_2$};     
	    \draw[red] (0,3.5) node {$P$} ;
    \draw[blue] (1.5,1.5) node {$Q$} ;
\end{tikzpicture}
\caption{A pair of paths with $\cro(P,Q)=3$.}
\label{fig:pair}
\end{figure}

The next theorem enumerates such pairs of paths. Note that, when there is no requirement on the number of crossings, the enumeration is trivial, since
\begin{equation}
\label{eq:m=0noq}
\left|\paths{A_1}{B_\de}{A_2}{B_\bde}\right|=
\left|\P_{A_1\to B_\de}\right|\cdot\left|\P_{A_2\to B_\bde}\right|
=\binom{u_\de-x_1+v_\de-y_1}{u_\de-x_1}\binom{u_\bde-x_2+v_\bde-y_2}{u_\bde-x_2},
\end{equation}
where $A_1=(x_1,y_1)$, $A_2=(x_2,y_2)$, $B_1=(u_1,v_1)$ and $B_2=(u_2,v_2)$.
We use the notation $A_1\prec A_2$ to indicate that $A_1$ is strictly northwest of $A_2$, that is, $x_1<x_2$ and $y_1>y_2$.

\begin{theorem}\label{thm:pairs}
Let $A_1=(x_1,y_1)$, $A_2=(x_2,y_2)$, $B_1=(u_1,v_1)$, $B_2=(u_2,v_2)$ be points in $\Z^2$, where $A_1\prec A_2$ and $B_1\prec B_2$.
Then, for all $m\ge0$,
\begin{equation}\label{eq:switched-noq}
\left|\pathsP{A_1}{B_2}{A_2}{B_1}{2m+1}\right|=\left|\pathsP{A_1}{B_2}{A_2}{B_1}{2m}\right|=\binom{u_2-x_1+v_2-y_1}{u_2-x_1+2m}\binom{u_1-x_2+v_1-y_2}{u_1-x_2-2m},
\end{equation}
and for all $m\ge1$,
\begin{equation}\label{eq:same-noq}
\left|\pathsP{A_1}{B_1}{A_2}{B_2}{2m}\right|=\left|\pathsP{A_1}{B_1}{A_2}{B_2}{2m-1}\right|
=\binom{u_2-x_1+v_2-y_1}{u_2-x_1+2m-1}\binom{u_1-x_2+v_1-y_2}{u_1-x_2-2m+1}.
\end{equation}
Let now $A=(x,y)$ and $B=(u,v)$ be points in $\Z^2$. Then, for all $r\ge0$,
\begin{align}
\label{eq:A1=A2-noq}
\left|\pathsP{A}{B_1}{A}{B_2}{r}\right|&=\binom{u_2-x+v_2-y}{u_2-x+r}\binom{u_1-x+v_1-y}{u_1-x-r},\\
\label{eq:B1=B2-noq}
\left|\pathsP{A_1}{B}{A_2}{B}{r}\right|&=\binom{u-x_1+v-y_1}{u-x_1+r}\binom{u-x_2+v-y_2}{u-x_2-r},\\
\label{eq:A1=A2,B1=B2-noq}
\left|\pathsP{A}{B}{A}{B}{r}\right|&
=\begin{cases}\displaystyle\binom{u-x+v-y}{u-x}^2& \text{if }r=0,\\
\displaystyle2\sum_{j\ge1}(-1)^{j-1}\binom{u-x+v-y}{u-x+r+j}\binom{u-x+v-y}{u-x-r-j}
& \text{if }r\ge1.
\end{cases}
\end{align}
\end{theorem}

The only case in which the formula given by Theorem~\ref{thm:pairs} is not a product of binomial coefficients is when both endpoints of the paths coincide, i.e., Equation~\eqref{eq:A1=A2,B1=B2-noq}. An alternative expression for this case, with a different number of summands, 
will be provided in Equation~\eqref{eq:A1=A2,B1=B2-noq-alt}.

As we will see in Section~\ref{sec:noq}, Theorem~\ref{thm:pairs} can be proved using a bijection that repeatedly swaps the prefixes of the paths up until, and including the step right after, the first crossing.
This is similar to the prefix-swapping method in the standard proof of the Lindstr\"om--Gessel--Viennot determinantal formula counting non-intersecting tuples of paths~\cite{Lin,GV}.

As in Section~\ref{sec:line}, we are interested in the refined enumeration by the major index. In this case, the relevant statistic is the sum of the major indices of the two paths, which we refer to as the {\em total major index}.
For $A_1,A_2,B_1,B_2\in\Z^2$ and $r\ge0$, define the polynomials
$$H^{\ge r}_{A_1\to B_\de,A_2\to B_\bde}(q)= \sum_{(P,Q)\in\pathsP{A_1}{B_\de}{A_2}{B_\bde}{r}} q^{\maj(P)+\maj(Q)}.$$
The polynomials counting pairs of paths that cross each other exactly $r$ times can be obtained from these as
$$H^{=r}_{A_1\to B_\de,A_2\to B_\bde}(q)=H^{\ge r}_{A_1\to B_\de,A_2\to B_\bde}(q)-H^{\ge r+1}_{A_1\to B_\de,A_2\to B_\bde}(q).$$

To state the expressions for these polynomials, it is convenient to define the following function of $A_1=(x_1,y_1)$, $A_2=(x_2,y_2)$, $B_1=(u_1,v_1)$, $B_2=(u_2,v_2)$, and $r$: 
\begin{equation}\label{eq:fr}
f_{r,A_1,A_2,B_2,B_1}(q)=q^{r(r+x_2-x_1)}\qbin{u_2-x_1+v_2-y_1}{u_2-x_1+r}\qbin{u_1-x_2+v_1-y_2}{u_1-x_2-r}.
\end{equation}
When there is no requirement on the number of crossings, Lemma~\ref{lem:qbin} immediately gives
\begin{equation}\label{eq:m=0}
H^{\ge 0}_{A_1\to B_\de,A_2\to B_\bde}(q)=\qbin{u_\de-x_1+v_\de-y_1}{u_\de-x_1}\qbin{u_\bde-x_2+v_\bde-y_2}{u_\bde-x_2}=f_{0,A_1,A_2,B_\de,B_\bde}(q)
\end{equation}
for arbitrary endpoints, since the two paths can be chosen independently.

In addition to the hypotheses from Theorem~\ref{thm:pairs}, the refinement by major index requires that the initial points $A_1$ and $A_2$ lie on the same line of slope $-1$.

\begin{theorem}\label{thm:pairs_refined}
Let $A_1=(x_1,y_1)$, $A_2=(x_2,y_2)$, $B_1=(u_1,v_1)$ and $B_2=(u_2,v_2)$ be points in $\Z^2$, where $A_1\prec A_2$ and $B_1\prec B_2$. Suppose additionally that 
\begin{equation}
\label{condition}
x_1+y_1=x_2+y_2.
\end{equation} 
Then, for all $m\ge0$,
\begin{equation}\label{eq:switched}
H^{\ge 2m+1}_{A_1\to B_2,A_2\to B_1}(q)=H^{\ge 2m}_{A_1\to B_2,A_2\to B_1}(q)=f_{2m,A_1,A_2,B_2,B_1}(q),
\end{equation}
and for all $m\ge1$,
\begin{equation}\label{eq:same}
H^{\ge 2m}_{A_1\to B_1,A_2\to B_2}(q)=H^{\ge 2m-1}_{A_1\to B_1,A_2\to B_2}(q)=f_{2m-1,A_1,A_2,B_2,B_1}(q).
\end{equation}
Let now $A=(x,y)$ and $B=(u,v)$ be points in $\Z^2$. Then, for all $r\ge0$,
\begin{align}\label{eq:A1=A2}
H^{\ge r}_{A\to B_1,A\to B_2}(q)&=f_{r,A,A,B_2,B_1}(q),\\
\label{eq:B1=B2}
H^{\ge r}_{A_1\to B,A_2\to B}(q)&=f_{r,A_1,A_2,B,B}(q),\\
\label{eq:A1=A2,B1=B2}
H^{\ge r}_{A\to B,A\to B}(q)&=\begin{cases}f_{0,A,A,B,B}(q) & \text{if }r=0,\\
2\sum_{j\ge1}(-1)^{j-1}f_{r+j,A,A,B,B}(q) & \text{if }r\ge1.
\end{cases}
\end{align}
\end{theorem}

All the formulas in Theorem~\ref{thm:pairs_refined} consist of a product of two $q$-binomial coefficients and a power of $q$,
with the exception of Equation~\eqref{eq:A1=A2,B1=B2} for $r\ge1$. An alternative expression for this case will be given in Equation~\eqref{eq:A1=A2,B1=B2-alt}.

Note that condition~\eqref{condition} is equivalent to $x_2-x_1=y_1-y_2$, and to the fact that $A_1$ and $A_2$ lie on the same line of slope $-1$. When it holds, the term $q^{r(r+x_2-x_1)}$ in Equation~\eqref{eq:fr} can also be written as $q^{r(r+y_1-y_2)}$.

Similarly to how the argument based on the iterated reflection principle for paths crossing a line does not give a proof of Theorems~\ref{thm:xaxis} and~\ref{thm:line} with the refinement by major index, the argument based on iterated prefix-swapping that can be used to prove Theorem~\ref{thm:pairs} does not give a proof of Theorem~\ref{thm:pairs_refined}. This is because the statistic $\maj$ does not behave well when swapping fragments of the paths.

\section{Proof ingredients: the bijections $\ta$, $\si$, $\bt$, $\bs$}\label{sec:ingredients} 

The proofs of the theorems Section~\ref{sec:main} rely on repeated applications of certain bijections that we describe next. An important feature of these bijections is that they allow us to keep track of the changes in the major index of the paths. 

We first describe the bijections in terms of paths with $N$ and $E$ steps.
Let $\P_{A\to B}^E$ and $\P_{A\to B}^N$ denote the subsets of $\P_{A\to B}$ consisting of paths that end in $E$ and $N$, respectively. If we do not specify the endpoints, the union of these sets over all possible endpoints $A,B\in\Z^2$ will be denoted by $\P^E$ and $\P^N$, respectively.

Let $\rho$ be the involution on paths with $N$ and $E$ steps induced by reflecting $\Z^2$ along the diagonal $y=x$, so that the coordinates of each point are switched. Clearly, $\rho$ is a bijection between $\P_{(x,y)\to(u,v)}^E$ and $\P_{(y,x)\to(v,u)}^N$.

Let $A=(x,y)$ and $B=(u,v)$. A key observation is that paths in $\P_{A\to B}$ are uniquely determined by the coordinates of their valleys. 
Specifically, there exists a path in $\P_{A\to B}$ whose valleys are at coordinates $(c_1,d_1), (c_2,d_2), \dots, (c_k,d_k)$ if and only if 
$$x< c_1<c_2<\dots<c_k\le u \quad\text{and}\quad y\le d_1<d_2<\dots<d_k< v.$$
Additionally, such a path ends in $E$ if and only if $c_k<u$.

Similarly, paths in $\P_{A\to B}$ are uniquely determined by the coordinates of their peaks. 
There exists a path in $\P_{A\to B}$ whose peaks are at coordinates $(c_1,d_1), (c_2,d_2), \dots, (c_k,d_k)$ if and only if 
$$x\le c_1<c_2<\dots<c_k<u \quad\text{and}\quad y< d_1<d_2<\dots<d_k\le v.$$ 
Such a path ends in $N$ if and only if $d_k<v$.

Define the vector $\vv=(1,-1)$, so that $A+\vv=(x+1,y-1)$ and $A-\vv=(x-1,y+1)$.
With the above considerations, we define a map 
$$\bt:\P_{A\to B}^E\to\P_{A+\vv\to B}^N,$$
as follows. Given $P\in\P_{A\to B}^E$, let $\bt(P)$ be the path in $\P_{A+\vv\to B}^N$ whose peaks are precisely at the coordinates of the valleys of $P$. See the examples in Figures~\ref{fig:btbs} and~\ref{fig:tasirho}. 
Similarly, we define a map 
$$\bs:\P_{A\to B}^N\to\P_{A-\vv\to B}^E$$
by letting 
$\bs(P)$ be the path in $\P_{A-\vv\to B}^E$ whose valleys are precisely at the coordinates of the peaks of $P\in\P_{A\to B}^N$.

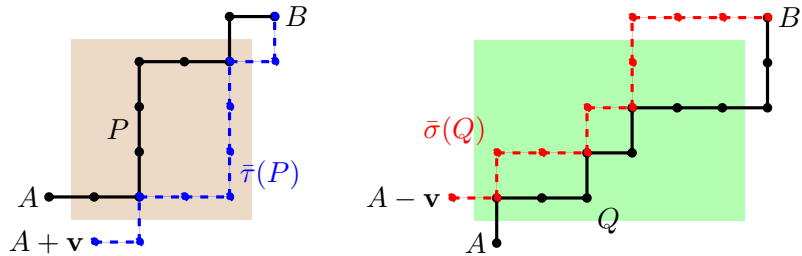
\begin{figure}[htb]
\begin{center}
 \begin{tikzpicture}[scale=0.6]
    \fill[brown!30] (.5,-.5) rectangle (4.5,3.5);
    \draw[very thick,fill](0,0) \start\E\E\N\N\N\E\E\N\E;
    \draw (2,1.5) node[left] {$P$};
    \draw[blue] (4,0.5) node[right] {$\bt(P)$};
    \draw[blue,dashed,very thick,fill](1,-1)  \start\E\N\E\E\N\N\N\E\N;
	\draw (0,0) node[left] {$A$};
	\draw (1,-1) node[left] {$A+\vv$};   
	\draw (5,4) node[right] {$B$};
\end{tikzpicture}
\quad
\begin{tikzpicture}[scale=0.6]
    \fill[green!30] (-.5,.5) rectangle (5.5,4.5);
    \draw[very thick,fill](0,0) \start\N\E\E\N\E\N\E\E\E\N\N;
    \draw (2,.5) node[right] {$Q$};
    \draw[red] (0,2.5) node[left] {$\bs(Q)$};
    \draw[red,dashed,very thick,fill](-1,1) \start\E\N\E\E\N\E\N\N\E\E\E;
	\draw (0,0) node[left] {$A$};
	\draw (-1,1) node[left] {$A-\vv$};       
	\draw (6,5) node[right] {$B$};
\end{tikzpicture}
\end{center}
\caption{Examples of the bijections $\bt$ and $\bs$. The brown rectangle (left) is the region of the possible locations of the valleys of paths in $\P_{A\to B}^E$, equivalently the peaks of paths in $\P_{A+\vv\to B}^N$.
The green rectangle (right) shows the possible locations of the peaks of paths in $\P_{A\to B}^N$, equivalently the valleys of paths in $\P_{A-\vv\to B}^E$.}
\label{fig:btbs}
\end{figure}

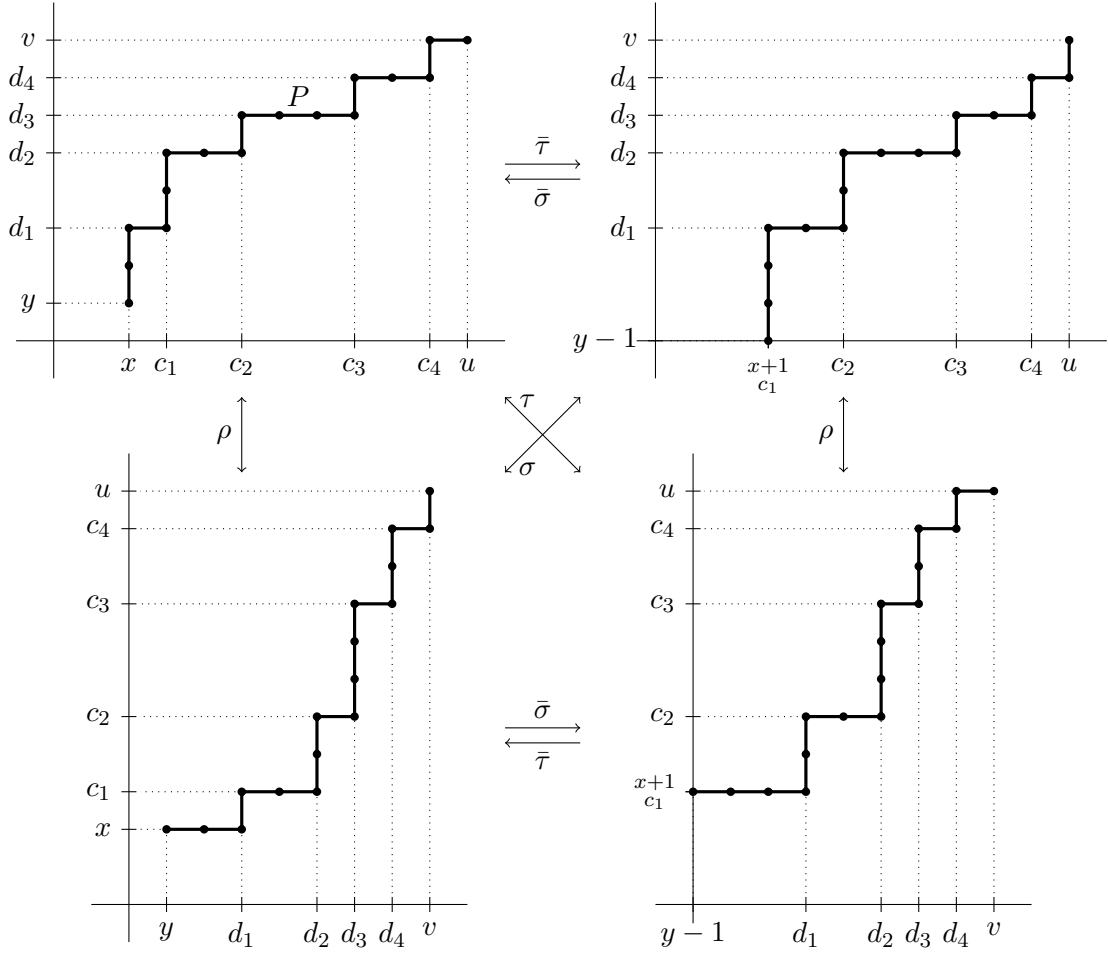
\begin{figure}[htb]
\centering
\begin{tikzpicture}[scale=.5]
\draw (-1,0)--(12,0);
\draw (0,-1)--(0,9);
\draw[very thick,fill](2,1) \start\N\N\E\N\N\E\E\N\E\E\E\N\E\E\N\E;
\markx{2}{1}{x}{y};
\markx{3}{3}{c_1}{d_1};
\markx{5}{5}{c_2}{d_2};
\markx{8}{6}{c_3}{d_3};
\markx{10}{7}{c_4}{d_4};
\markx{11}{8}{u}{v};
\draw (6.5,6.5) node {$P$};
\draw[->] (12,4.7)-- node[above]{$\bt$} (14,4.7);
\draw[<-] (12,4.3)-- node[below]{$\bs$} (14,4.3);
\draw[<->] (5,-1.5)-- node[left]{$\rho$} (5,-3.5);
\draw[<->] (12,-1.5)-- (14,-3.5); \draw (12.6,-1.6) node {$\ta$};
\draw[<->] (14,-1.5)-- (12,-3.5); \draw (12.6,-3.4) node {$\si$};

\begin{scope}[shift={(16,0)}] 
\draw (-.5,0)--(12,0);
\draw (0,-1)--(0,9);
\draw[very thick,fill](3,0) \start	\N\N\N\E\E\N\N\E\E\E\N\E\E\N\E\N;
\markx{3}{0}{\substack{x+1\\ c_1}}{y-1};
\markx{3}{3}{}{d_1};
\markx{5}{5}{c_2}{d_2};
\markx{8}{6}{c_3}{d_3};
\markx{10}{7}{c_4}{d_4};
\markx{11}{8}{u}{v};
\draw[<->] (5,-1.5)-- node[left]{$\rho$} (5,-3.5);
\end{scope}

\begin{scope}[shift={(2,-15)}] 
\draw (-1,0)--(9,0);
\draw (0,-1)--(0,12);
\draw[very thick,fill](1,2) \start\E\E\N\E\E\N\N\E\N\N\N\E\N\N\E\N;
\marky{2}{1}{x}{y};
\marky{3}{3}{c_1}{d_1};
\marky{5}{5}{c_2}{d_2};
\marky{8}{6}{c_3}{d_3};
\marky{10}{7}{c_4}{d_4};
\marky{11}{8}{u}{v};
\draw[->] (10,4.7)-- node[above]{$\bs$} (12,4.7);
\draw[<-] (10,4.3)-- node[below]{$\bt$} (12,4.3);
\end{scope}

\begin{scope}[shift={(17,-15)}] 
\draw (-1,0)--(9,0);
\draw (0,-.5)--(0,12);
\draw[very thick,fill](0,3) \start		\E\E\E\N\N\E\E\N\N\N\E\N\N\E\N\E;
\marky{3}{0}{\substack{x+1\\ c_1}}{y-1};
\marky{3}{3}{}{d_1};
\marky{5}{5}{c_2}{d_2};
\marky{8}{6}{c_3}{d_3};
\marky{10}{7}{c_4}{d_4};
\marky{11}{8}{u}{v};
\end{scope}

\end{tikzpicture}
\caption{An example of the relationships among the maps $\bt,\bs,\ta,\si,\rho$, illustrating their effect on the valleys of the paths.
The path $P$ in the upper left has valleys at $(c_i,d_i)$ for $1\le i\le 4$.}
\label{fig:tasirho}
\end{figure}

\begin{lemma}\label{lem:bsbt}
The maps
$$\bt:\P_{A\to B}^E\to\P_{A+\vv\to B}^N$$
and
$$\bs:\P_{A\to B}^N\to\P_{A-\vv\to B}^E$$
defined above are bijections. In addition, when viewed as maps between $\P^E$ and $\P^N$, $\bs$ and $\bt$ are inverses of each other, that is, $\bs(\bt(P))=P$ and $\bt(\bs(P))=P$ for all $P$ in the domain.
\end{lemma}

\begin{proof}
The fact that $\bt$ is well defined and it is a bijection follows by noting that the inequalities satisfied by the coordinates of the valleys of paths in $\P_{A\to B}^E$ coincide with those satisfied by the coordinates of the peaks of paths in $\P_{A+\vv\to B}^N$. Specifically, if $A=(x,y)$, $B=(u,v)$, and these coordinates are
$(c_1,d_1), (c_2,d_2), \dots, (c_k,d_k)$, they satisfy
$$x<c_1<c_2<\dots<c_k<u \quad\text{and}\quad y\le d_1<d_2<\dots<d_k< v,$$ 
or equivalently
$$x+1\le c_1<c_2<\dots<c_k<u \quad\text{and}\quad y-1< d_1<d_2<\dots<d_k< v.$$ 

Its inverse $\bt^{-1}:\P_{A+\vv\to B}^N\to\P_{A\to B}^E$ is the map that turns peaks into valleys, and so it coincides with our definition of $\bs$, with $A+\vv$ playing the role of $A$. This completes the proof.
\end{proof}

Two more maps closely related to $\bt$ and $\bs$ that will be useful in our constructions are $\ta=\rho\circ\bt$ and $\si=\rho\circ\bs$. The first one is a bijection
$$\ta:\P_{(x,y)\to(u,v)}^E\to\P_{(y-1,x+1)\to(v,u)}^E,$$ 
and it maps the path with valleys at $(c_1,d_1), \dots, (c_k,d_k)$ to the path with valleys at 
$(d_1,c_1), \dots,$ $(d_k,c_k)$. Clearly, viewed as a map from $\P^E$ to itself, $\ta$ is an involution, in the sense that $\ta^{-1}=\ta$.

The second one is a bijection
$$\si:\P_{(x,y)\to(u,v)}^N\to\P_{(y+1,x-1)\to(v,u)}^N,$$
and it maps the path with peaks at $(c_1,d_1), \dots, (c_k,d_k)$ to the path with peaks at 
$(d_1,c_1),\dots, (d_k,c_k)$.
Again, $\si$ is an involution of $\P^N$.

See Figure~\ref{fig:tasirho} for an example of these bijections, and Figure~\ref{fig:tasi-diagram} for a diagram of their relationships. The next lemma describes how each of the maps affects the major index.

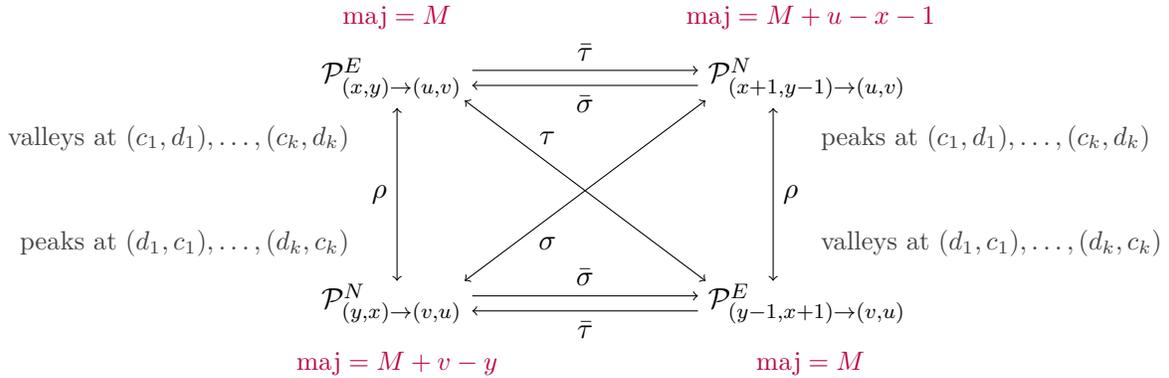
\begin{figure}[htb]
\centering
\begin{tikzpicture}[scale=1]
\draw (0,0) node[left] {$\P_{(x,y)\to(u,v)}^E$};
\draw (3,0) node[right] {$\P_{(x+1,y-1)\to(u,v)}^N$};
\draw (0,-3) node[left] {$\P_{(y,x)\to(v,u)}^N$};
\draw (3,-3) node[right] {$\P_{(y-1,x+1)\to(v,u)}^E$};
\draw[->] (0,.1) -- node[above] {$\bt$} (3,.1);
\draw[<-] (0,-.1) -- node[below] {$\bs$} (3,-.1);
\draw[->] (0,-2.9) -- node[above] {$\bs$} (3,-2.9);
\draw[<-] (0,-3.1) -- node[below] {$\bt$} (3,-3.1);
\draw[<->] (-.1,-.3) -- (3.1,-2.7);
\draw(1,-.8) node {$\ta$};
\draw[<->] (-.1,-2.7) -- (3.1,-.3);
\draw(1,-2.2) node {$\si$};
\draw[<->] (-1,-.4) -- node[left] {$\rho$} (-1,-2.7);
\draw[<->] (4,-.4) -- node[right] {$\rho$} (4,-2.7);

\draw[purple] (-1,.8) node {\small $\maj=M$};
\draw[purple] (4.5,.8) node {\small $\maj=M+u-x-1$};
\draw[purple] (-1,-3.8) node {\small $\maj=M+v-y$};
\draw[purple] (4.5,-3.8) node {\small $\maj=M$};

\draw[darkgray] (-1.5,-.8) node[left] {\small valleys at $(c_1,d_1),\dots,(c_k,d_k)$};
\draw[darkgray]  (4.5,-.8) node[right] {\small peaks at $(c_1,d_1),\dots,(c_k,d_k)$};
\draw[darkgray]  (-1.5,-2.2) node[left] {\small peaks at $(d_1,c_1),\dots,(d_k,c_k)$};
\draw[darkgray]  (4.5,-2.2) node[right] {\small valleys at $(d_1,c_1),\dots,(d_k,c_k)$};
\end{tikzpicture}
\caption{Diagram of the relationships among the maps $\ta,\si,\bt,\bs$, and their effect on $\maj$.}
\label{fig:tasi-diagram}
\end{figure}

\begin{lemma}\label{lem:maj}
If $P\in \P_{(x,y)\to(u,v)}$, then
\begin{subnumcases}{\maj(\rho(P))= \label{eq:majrho}} 
 \maj(P)+v-y  & if $P$ ends in $E$, \label{eq:majrhoE}\\
 \maj(P)-u+x  & if $P$ ends in $N$. \label{eq:majrhoN}
\end{subnumcases}
If $P\in\P_{(x,y)\to(u,v)}^E$, then 
\begin{align}
\label{eq:majbt} \maj(\bt(P))&=\maj(P)+u-x-1,\\
\label{eq:majta} \maj(\ta(P))&=\maj(P).
\end{align}
If $P\in\P_{(x,y)\to(u,v)}^N$, 
then 
\begin{align}
\label{eq:majbs} \maj(\bs(P))&=\maj(P)-u+x,\\
\label{eq:majsi} \maj(\si(P))&=\maj(P)-u+x+v-y-1.
\end{align}
\end{lemma}

\begin{proof}
Let $P\in\P_{(x,y)\to(u,v)}$, and suppose that its valleys have coordinates $(c_1,d_1),  \dots, (c_k,d_k)$. Then
\begin{equation}
\label{eq:maj-original}
\maj(P)=\sum_{i=1}^k(c_i+d_i-x-y)=\sum_{i=1}^k(c_i+d_i)-k(x+y).
\end{equation}
The path $\rho(P)\in\P_{(y,x)\to(v,u)}$ has peaks at $(d_1,c_1),\dots,(d_k,c_k)$. 

If $P$ ends in $E$, then $\rho(P)$ ends in $N$, and so $\rho(P)$ has valleys at $(d_2,c_1), (d_3,c_2),\dots,(d_k,c_{k-1})$, $(v,c_k)$, with an additional valley at $(d_1,x)$ if and only if $\rho(P)$ starts with an $E$, which happens precisely when $y<d_1$ (as in the example on the left of Figure~\ref{fig:tasirho}). In both cases, noting that the starting point of $\rho(P)$ is $(y,x)$, we have
\begin{align*}
\maj(\rho(P))&=(d_1+x-y-x)+\sum_{i=1}^{k-1}(d_{i+1}+c_i-y-x)+(v+c_k-y-x)\\ 
&=\sum_{i=1}^k(c_i+d_i)-k(x+y)+v-y\\
&=\maj(P)+v-y. 
\end{align*}
Indeed, even if $\rho(P)$ does not have a valley at $(d_1,x)$, then $d_1=y$, in which case the term $d_1+x-y-x$ does not contribute to the major index, so the above formula is still valid. This proves Equation~\eqref{eq:majrhoE}.

If $P$ ends in $N$, then $\rho(P)$ ends in $E$, so we can apply Equation~\eqref{eq:majrhoE} to the path $\rho(P)\in\P_{(y,x)\to(v,u)}^E$. We obtain
$$\maj(P)=\maj(\rho(\rho(P)))=\maj(\rho(P))+u-x,$$
from where $\maj(\rho(P))=\maj(P)-u+x$, proving Equation~\eqref{eq:majrhoN}.

Now suppose again that $P\in\P_{(x,y)\to(u,v)}^E$. The valleys of $\ta(P)$ have coordinates $(d_1,c_1), \dots,(d_k,c_k)$, and the coordinates of the starting point of $\ta(P)$ sum to $(y-1)+(x+1)=x+y$. It follows that $\maj(\ta(P))$ coincides with the right-hand side of Equation~\eqref{eq:maj-original}. This proves Equation~\eqref{eq:majta}.

Equation~\eqref{eq:majbt} can be proved with an argument similar to the proof of Equation~\eqref{eq:majrho}. Alternatively, it can be deduced from this equation using the fact that $\bt(P)=\rho^{-1}(\ta(P))=\rho(\ta(P))$. Since $\ta(P)\in\P_{(y-1,x+1)\to(v,u)}^E$, Equations~\eqref{eq:majrho} and~\eqref{eq:majta} give
$$\maj(\bt(P))= \maj(\rho(\ta(P)))=\maj(\ta(P))+u-(x+1)=\maj(P)+u-x-1.$$

To deduce Equation~\eqref{eq:majbs} from Equation~\eqref{eq:majbt}, suppose now
 that $P\in\P_{(x,y)\to(u,v)}^N$.  Since $P=\bt(\bs(P))$ by Lemma~\ref{lem:bsbt}, applying Equation~\eqref{eq:majbt} to the path $\bs(P)\in\P_{(x-1,y+1)\to(u,v)}^E$ gives
$$\maj(P)=\maj(\bs(P))+u-(x-1)-1,$$
proving Equation~\eqref{eq:majbs}.

Finally, to prove Equation~\eqref{eq:majsi}, we use the fact that $\si(P)=\rho(\bs(P))$. Since $\bs(P)\in\P_{(x-1,y+1)\to(u,v)}^E$, Equations~\eqref{eq:majrho} and~\eqref{eq:majbs} give
\[
\maj(\si(P))= \maj(\rho(\bs(P)))=\maj(\bs(P))+v-(y+1)=\maj(P)-u+x+v-y-1.\qedhere
\]
\end{proof}

Let us illustrate Lemma~\ref{lem:maj} with some examples. Note that the quantities $u-x$ and $v-y$ are simply the number of $E$ and $N$ steps of $P$, respectively. 
If $P$ is the path in the top left of Figure~\ref{fig:tasirho}, then $\maj(P)=3+7+11+14=35$, whereas the path $\rho(P)$ in the bottom left has $\maj(\rho(P))=2+5+8+12+15=42$, so applying $\rho$ increases the major index by $v-y=7$. The path $\bt(P)$ in the top right has $\maj(\bt(P))=5+10+13+15=43$, so $\bt$ increases the major index by $u-x-1=8$. The path $\ta(P)$ in the bottom right has major index $\maj(\ta(P))=3+7+11+14=35=\maj(P)$.

Via the straightforward correspondence described in Section~\ref{sec:basic} between paths with $N$ and $E$ steps and paths with $U$ and $D$ steps, we can interpret all the maps in this section as maps on sets of the form $\G_{a,b}$, where $a,b\ge0$. For example, denoting by $\G_{a,b}^D$ and $\G_{a,b}^U$ the subsets of $\G_{a,b}$ consisting of paths that end in $D$ and $U$, respectively,
we can view $\ta$ and $\si$ as maps
\begin{equation}\label{eq:tasi}
\ta:\G_{a,b}^D\to\G_{b-1,a+1}^D \quad\text{and}\quad \si:\G_{a,b}^U\to\G_{b+1,a-1}^U.
\end{equation}
It follows from Lemma~\ref{lem:maj} that, if $P\in \G_{a,b}^D$, then
\begin{equation}\label{eq:majta2}
\maj(\ta(P))=\maj(P),
\end{equation}
and if $P\in \G_{a,b}^U$, then
\begin{equation}\label{eq:majsi2}
\maj(\si(P))=\maj(P)+a-b-1.
\end{equation}
Additionally, if $P\in \G_{a,b}$, then
\begin{equation}\label{eq:majrho2}
 \maj(\rho(P))=\begin{cases} 
 \maj(P)+a & \text{if $P$ ends in $D$},\\
 \maj(P)-b & \text{if $P$ ends in $U$}.
\end{cases}
\end{equation}
These maps will play a key role in the next section.

\section{Proofs for paths crossing a line} 
\label{sec:line-proofs}

The goal of this section is to prove Theorems~\ref{thm:xaxis} and~\ref{thm:line}.
Our bijections will be easier to visualize if we allow the starting point of the lattice paths with $U$ and $D$ steps to be anywhere on the $y$-axis, by identifying paths in $\G_{a,b}$ with their vertical translations. 
In particular, it will be convenient to identify $\G_{a,b}^{\ge r,\ell}$ with the set of paths with $a$ steps $U$ and $b$ steps $D$ that start at the point $(0,-\ell)$ and cross the $x$-axis at least $r$ times. Note that the ending point of such paths is $(a+b,-\ell+a-b)$, and that vertical translations do not affect the major index.

In a similar fashion, by applying vertical translations as needed, we will interpret the domain and the range of the maps $\ta$ and $\si$ from Equation~\eqref{eq:tasi} as consisting of paths that end on the $x$-axis. With this perspective, for $P\in\G_{a,b}^D$, viewed as a path starting at $(0,b-a)$ and ending at $(a+b,0)$, its image $\ta(P)$ is the path starting at $(0,a-b+2)$, ending at the same point $(a+b,0)$, and whose valleys are obtained by reflecting the valleys of $P$ along the $x$-axis. 
 Similarly, for $P\in\G_{a,b}^U$ starting at $(0,b-a)$ and ending at $(a+b,0)$, its image $\si(P)$ is the path 
starting at $(0,a-b-2)$, ending at the same point $(a+b,0)$, and whose peaks are obtained by reflecting the peaks of $P$ along the $x$-axis. We will use these convenient descriptions of $\ta$ and $\si$ throughout this section. See Figure~\ref{fig:tasi-x} for examples.

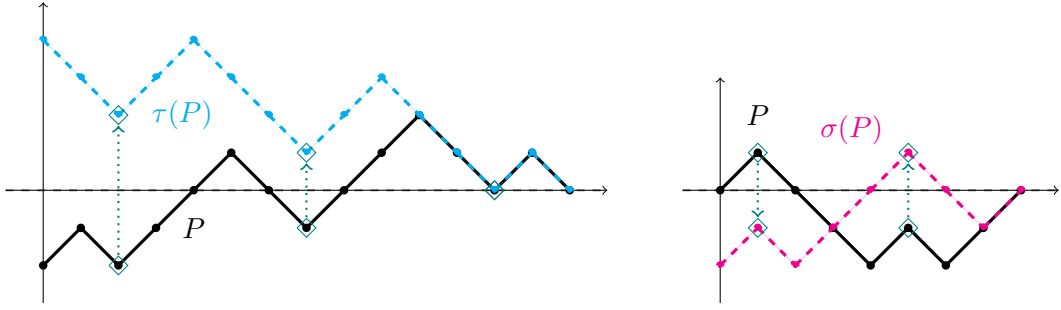
\begin{figure}[htb]
\centering
\begin{tikzpicture}[scale=0.5]
\draw[dashed,thick,gray] (-1,0)--(15,0);
\draw[->] (-1,0)--(15,0);
\draw[->] (0,-3)--(0,5);
\draw[very thick,fill] (0,-2) \startud \up\dn\up\up\up\dn\dn\up \up\up\dn\dn\up\dn;
\draw (4,-1) node {$P$};
\draw[cyan,dashed,very thick,fill](0,4)  \startud\dn\dn\up\up\dn\dn\dn\up\up\dn\dn\dn\up\dn;
\draw[cyan] (3.7,2) node {$\ta(P)$};
\movepeakvalley{2}{-2}
\movepeakvalley{7}{-1}
\movepeakvalley{12}{0}

\begin{scope}[shift={(18,0)}]
\draw[dashed,thick,gray] (-1,0)--(9,0);
\draw[->] (-1,0)--(9,0);
\draw[->] (0,-3)--(0,3);
\draw[very thick,fill] (0,0) \startud \up\dn\dn\dn\up\dn\up\up;
\draw (1,2) node {$P$};
\draw[magenta,dashed,very thick,fill](0,-2)  \startud\up\dn\up\up\up\dn\dn\up;
\draw[magenta] (3.5,1.6) node {$\si(P)$};
\movepeakvalley{1}{1}
\movepeakvalley{5}{-1}
\end{scope}
\end{tikzpicture}
\caption{Examples of the bijections $\ta$ and $\si$ on paths that have been translated vertically so that they end on the $x$-axis.}
\label{fig:tasi-x}
\end{figure}

Throughout this section, we will assume that $a,b,r\ge0$ and $\ell\in\Z$.

\begin{definition}\label{def:tarsir}
Given $P\in\G^{\ge r,\ell}_{a,b}$, viewed as a path from $(0,-\ell)$ to $(a+b,-\ell+a-b)$ crossing the $x$-axis at least $r$ times, label these crossings so that $C_j$ denotes the $j$th crossing from the right, for $1\le j\le r$. Decompose $P$ as $P=P_\lt P_\rt$ by splitting at $C_r$. 
If $C_r$ is a downward crossing, define 
$$\ta_r(P)=\ta(P_\lt)P_\rt.$$
If $C_r$ is an upward crossing, define 
$$\si_r(P)=\si(P_\lt)P_\rt.$$
\end{definition}

Examples of the maps $\ta_r$ and $\si_r$ are given in Figure~\ref{fig:tarsir}.

\begin{figure}[htb]
\centering
\begin{tikzpicture}[scale=0.5]
\draw[dashed,thick,gray] (-.5,0)--(15,0);
\draw[->] (-.5,0)--(15,0);
\draw[->] (0,-2)--(0,2);
\draw[very thick,fill] (0,-1) \startud \dn\up\dn\up\up\up\dn\up\dn \dn\up\up\up\dn;
\draw (7,1.5) node {$P$};
\crossing{5}{0}
\crossing{9}{0}
\crossing{11}{0}
\draw (5.2,-.8) node {$C_3$};
\draw (8.8,-.8) node {$C_2$};
\draw (11.2,-.8) node{$C_1$};

\begin{scope}[shift={(17,3)}]
\draw[dashed,thick,gray] (-.5,0)--(15,0);
\draw[->] (-.5,0)--(15,0);
\draw[->] (0,-2)--(0,4);
\draw[fill] (0,-1) \startud \dn\up\dn\up\up\up\dn\up\dn \dn\up\up\up\dn;
\draw[very thick,fill,cyan] (0,3) \startud \dn\up\dn\up\dn\dn\dn\up\dn \dn\up\up\up\dn;
\draw[cyan] (6,2.5) node {$\ta_2(P)$};
\movepeakvalley{1}{-2}
\movepeakvalley{3}{-2}
\movepeakvalley{7}{0}
\crossing{9}{0}
\draw (8.8,-.8) node {$C_2$};
\end{scope}

\begin{scope}[shift={(17,-3)}]
\draw[dashed,thick,gray] (-.5,0)--(15,0);
\draw[->] (-.5,0)--(15,0);
\draw[->] (0,-2)--(0,2);
\draw[fill] (0,-1) \startud \dn\up\dn\up\up\up\dn\up\dn \dn\up\up\up\dn;
\draw[very thick,fill,magenta] (0,-1) \startud \up\up\dn\dn\dn\up\dn\up\dn\up\up \up\up\dn;
\draw[magenta] (2.5,1.8) node {$\si_1(P)$};
\movepeakvalley{2}{-1}
\movepeakvalley{6}{1}
\movepeakvalley{8}{1}
\crossing{11}{0}
\draw (11.2,-.8) node{$C_1$};
\end{scope}
\end{tikzpicture}
\caption{The bijections $\ta_2$ and $\si_1$ applied to the path $P$ from Figure~\ref{fig:Gstrl}, which has been translated so that the line being crossed is the $x$-axis.}
\label{fig:tarsir}
\end{figure}
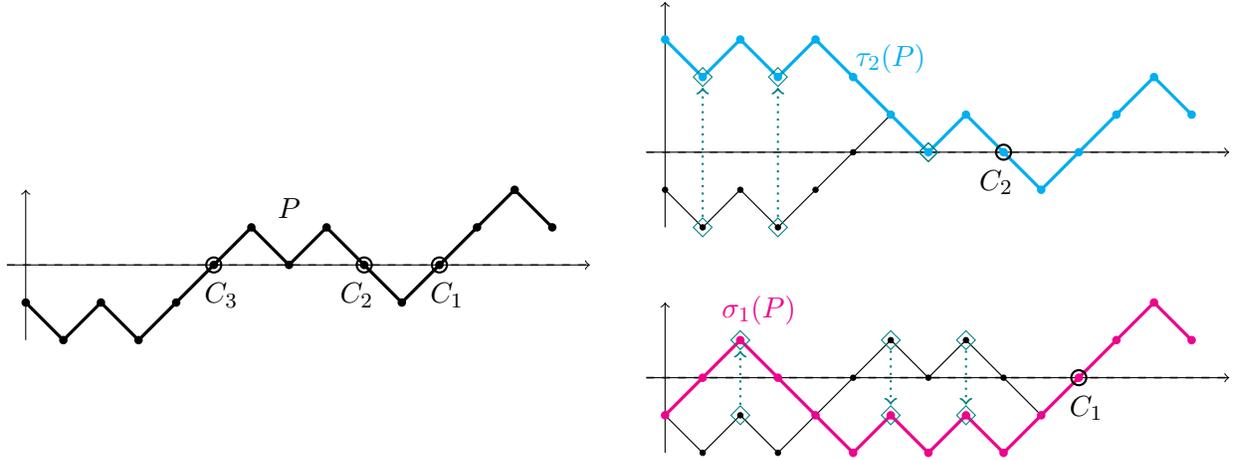

\begin{lemma}\label{lem:maj-tarsir}
Let $P\in\G^{\ge r,\ell}_{a,b}$. If $C_r$ is a downward crossing, then
\begin{equation}\label{eq:tar}
\ta_r(P)\in\G^{\ge r,-\ell-2}_{a-\ell-1,b+\ell+1}
\end{equation}
and $$\maj(\ta_r(P))=\maj(P).$$
Additionally, $\ta_r$ is a bijection between the subsets of $\G^{\ge r,\ell}_{a,b}$ and $\G^{\ge r,-\ell-2}_{a-\ell-1,b+\ell+1}$ consisting of paths whose $r$th crossing from the right is a downward crossing.

If $C_r$ is an upward crossing, then
\begin{equation}\label{eq:sir}
\si_r(P)\in\G^{\ge r,-\ell+2}_{a-\ell+1,b+\ell-1}
\end{equation}
and $$\maj(\si_r(P))=\maj(P)+\ell-1.$$
Additionally, $\si_r$ is a bijection between the subsets of $\G^{\ge r,\ell}_{a,b}$ and $\G^{\ge r,-\ell+2}_{a-\ell+1,b+\ell-1}$ consisting of paths whose $r$th crossing from the right is an upward crossing.
\end{lemma}

\begin{proof}
The path prefix $P_\lt$ starts at $(0,-\ell)$ and ends at $C_r$, which is on the $x$-axis. Suppose first that $C_r$ is a downward crossing.
Then $\ta(P_\lt)$ starts at $(0,\ell+2)$ and ends at the same point $C_r$. Thus, $\ta_r(P)$ is a path from $(0,\ell+2)$ to $(a+b,-\ell+a-b)$ that crosses the $x$-axis at least $r$ times, proving~\eqref{eq:tar}. Since the map $\ta$ preserves the last step (which is a $D$), the $r$th crossing of $\ta_r(P)$ from the right is still a downward crossing, and in fact $\ta_r$ is a bijection between the stated subsets. Indeed, since $\ta$ is an involution, the inverse of $\ta_r$ is $\ta_r$ itself applied to paths in $\G^{\ge r,-\ell-2}_{a-\ell-1,b+\ell+1}$ whose $r$th crossing from the right is a downward crossing.

By Definition~\ref{def:tarsir} and Equation~\eqref{eq:majta2}, we have
$$\maj(\ta_r(P))=\maj(\ta(P_\lt))+\maj(P_\rt)=\maj(P_\lt)+\maj(P_\rt)=\maj(P).$$

Suppose now that $C_r$ is an upward crossing. Then $\si(P_\lt)$ starts at $(0,\ell-2)$ and ends at $C_r$, and so
$\si_r(P)$ is a path from $(0,\ell-2)$ to $(a+b,-\ell+a-b)$ that crosses the $x$-axis at least $r$ times, proving~\eqref{eq:sir}.
An analogous argument to the one used for $\ta_r$ shows that $\si_r$ is a bijection between the stated subsets.
 By Definition~\ref{def:tarsir} and Equation~\eqref{eq:majsi2}, noting that the change in $y$-coordinate from the first to the last point of $P_\lt$ is $\ell$, we have
\[
\maj(\si_r(P))=\maj(\si(P_\lt))+\maj(P_\rt)=\maj(P_\lt)+\ell-1+\maj(P_\rt)=\maj(P)+\ell-1.\qedhere
\]
\end{proof}

We now have all the tools to prove our formulas counting paths by the major index and the number of crossings of a horizontal line.

\begin{proof}[Proof of Theorems~\ref{thm:xaxis} and~\ref{thm:line}]
Let $a,b,r\ge0$ and $\ell\in\Z$. We will use both interpretations of the elements of $\G^{\ge r,\ell}_{a,b}$: as paths from $(0,0)$ to $(a+b,a-b)$ crossing the line $y=\ell$ at least $r$ times, and as paths from $(0,-\ell)$ to $(a+b,-\ell+a-b)$ crossing the $x$-axis at least $r$ times. In both cases, we call the line being crossed the {\em reference line}, and {\em crossings} refer to the points where the path crosses the reference line. Given a path in $\G^{\ge r,\ell}_{a,b}$, we let $C_j$ denote the $j$th crossing from the right, for $1\le j\le r$.

The proof is divided into nine cases depending on whether the paths start below ($0<\ell$), on ($0=\ell$), or above ($0>\ell$) the reference line, and whether they end below ($\ell>a-b$), on ($\ell=a-b$), or above ($\ell<a-b$) this line. In each case, the goal is to determine $G^{\ge r,\ell}_{a,b}(q)$ by finding a bijection between $\G^{\ge r,\ell}_{a,b}$ and some set of the form $\G_{a',b'}$, with no requirement on the number of crossings.

In Cases I--IV below, the paths neither start nor end on the reference line, and so the parity of the number of crossings is fixed: it is even or odd according to whether the two endpoints are on the same or on opposite sides of the line. Thus, we get equalities of the form $\G^{\ge 2m,\ell}_{a,b}=\G^{\ge 2m-1,\ell}_{a,b}$ (if the endpoints are on the same side) or $\G^{\ge 2m+1,\ell}_{a,b}=\G^{\ge 2m,\ell}_{a,b}$ (if they are on opposite sides). Additionally,
if the right endpoint of a path $P$ is above (respectively below) the reference line, then $C_j$ is an upward (resp.\ downward) crossing for odd $j$, and a downward (resp.\ upward) crossing for even~$j$.

\begin{list1}

\item {\bf Case I: $0<\ell<a-b$.} Since the number of crossings of each path must be odd in this case, we have $\G^{\ge 2m+1,\ell}_{a,b}=\G^{\ge 2m,\ell}_{a,b}$ for all $m\ge0$. The case $m=0$ is solved in Lemma~\ref{lem:qbin}, so we assume that $m\ge1$.
Since paths in $\G^{\ge 2m,\ell}_{a,b}$ end above the reference line, the crossing $C_{2m}$ in these paths must be a downward crossing. Thus, by Lemma~\ref{lem:maj-tarsir}, $\ta_{2m}$ is a bijection between $\G^{\ge 2m,\ell}_{a,b}$ and $\G^{\ge 2m,-\ell-2}_{a-\ell-1,b+\ell+1}$ which preserves the major index.
Paths in the image start and end above the reference line, and so  $\G^{\ge 2m,-\ell-2}_{a-\ell-1,b+\ell+1}=\G^{\ge 2m-1,-\ell-2}_{a-\ell-1,b+\ell+1}$. For paths in this set, $C_{2m}$  is again a downward crossing and $C_{2m-1}$ is an upward crossing. 

Applying Lemma~\ref{lem:maj-tarsir} again, $\si_{2m-1}$ is a bijection between $\G^{\ge 2m-1,-\ell-2}_{a-\ell-1,b+\ell+1}$ and $\G^{\ge 2m-1,\ell+4}_{a+2,b-2}=\G^{\ge 2m-2,\ell+4}_{a+2,b-2}$ that changes $\maj$ by $(-\ell-2)-1=-(\ell+3)$. The resulting paths, like those in the original set, start below and end above the reference line. Repeating the same argument, we obtain a composition of bijections 
\begin{equation}\label{eq:composition}
\si_1\circ\ta_2\circ\dots\circ\si_{2m-1}\circ\ta_{2m}:\G^{\ge 2m,\ell}_{a,b}\to \G^{\ge 0,\ell+4i}_{a+2m,b-2m}=\G_{a+2m,b-2m}
\end{equation}
with the property that, if $Q$ is the image of $P$, then
$$\maj(Q)=\maj(P)-(\ell+3)-(\ell+7)-\dots-(\ell+4i-1)=\maj(P)-m(2m+1+\ell).$$
See Figure~\ref{fig:caseI} for an example of this composition.
It follows that
$$G_{a,b}^{\ge 2m+1,\ell}(q)=G_{a,b}^{\ge 2m,\ell}(q)=q^{m(2m+1+\ell)}\sum_{Q\in \G_{a+2m,b-2m}}q^{\maj(Q)}=q^{m(2m+1+\ell)}\qbin{a+b}{a+2m},$$
by Lemma~\ref{lem:qbin}. This proves Equation~\eqref{eq:0<l<a-b}.

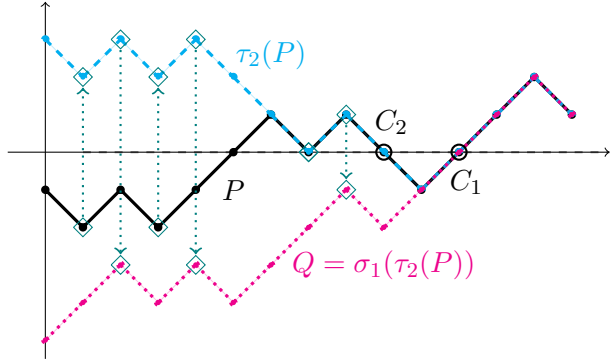
\begin{figure}[htb]
\centering
\begin{tikzpicture}[scale=0.5]
\draw[dashed,thick,gray] (1,0)--(15,0);
\draw[->] (-1,0)--(15,0);
\draw[->] (0,-5.5)--(0,4);
\draw[very thick,fill] (0,-1) \startud \dn\up\dn\up\up\up\dn\up\dn \dn\up\up\up\dn;
\draw (5,-1) node {$P$};
\draw[very thick,fill,cyan,dashed] (0,3) \startud \dn\up\dn\up\dn\dn\dn\up\dn \dn\up\up\up\dn;
\draw[cyan] (6,2.6) node {$\ta_2(P)$};
\draw[very thick,fill,magenta,dotted] (0,-5) \startud \up\up\dn\up\dn\up\up\up\dn\up\up\up\up\dn;
\draw[magenta] (9,-3) node {$Q=\si_1(\ta_2(P))$};
\movepeakvalley{1}{-2}
\movepeakvalley{3}{-2}
\movepeakvalley{7}{0}
\movepeakvalley{2}{3}
\movepeakvalley{4}{3}
\movepeakvalley{8}{1}
\crossing{9}{0}
\crossing{11}{0}
\draw (9.2,.8) node {$C_2$};
\draw (11.2,-.8) node{$C_1$};
\end{tikzpicture}
\caption{The composition $\si_1\circ\ta_2$ applied to the path $P\in\G_{8,6}^{\ge 3,1}$ from Figures~\ref{fig:Gstrl} and~\ref{fig:tarsir}.
Here $\maj(Q)=17=\maj(P)-(\ell+3)$.
}
\label{fig:caseI}
\end{figure}

An equivalent description of the bijection~\eqref{eq:composition} is obtained by repeatedly applying the maps $\si$ and $\ta$ to the appropriate path prefixes. Indeed, decomposing $P\in\G_{a,b}^{\ge 2m,\ell}$ as $P=P_0P_1P_2\dots P_{2m-1}P_{2m}$ by splitting at the rightmost $2m$ crossings of $P$ (so that each $P_j$ for $j\ge1$ lies entirely above or below the reference line), 
its image under this bijection is
$$\si(\ta(\cdots(\si(\ta(P_0)P_1)\cdots)P_{2m-2})P_{2m-1})P_{2m}.$$

\item {\bf Case II: $0>\ell>a-b$.}
This case is similar to Case I, with the roles of $\si$ and $\ta$ reversed. Again, $\G^{\ge 2m+1,\ell}_{a,b}=\G^{\ge 2m,\ell}_{a,b}$ for all $m\ge0$. For paths in $\G^{\ge 2m,\ell}_{a,b}$, where $m\ge1$, the crossing $C_{2m}$ must be an upward crossing. By Lemma~\ref{lem:maj-tarsir}, $\si_{2m}$ is a bijection between $\G^{\ge 2m,\ell}_{a,b}$ and $\G^{\ge 2m,-\ell+2}_{a-\ell+1,b+\ell-1}=\G^{\ge 2m-1,-\ell+2}_{a-\ell+1,b+\ell-1}$ the changes $\maj$ by $\ell-1$.
For paths in the image, which start and end below the reference line, $C_{2m}$ is an upward crossing and $C_{2m-1}$ is a downward crossing. 

Applying Lemma~\ref{lem:maj-tarsir} again, $\ta_{2m-1}$ is a bijection between $\G^{\ge 2m-1,-\ell+2}_{a-\ell+1,b+\ell-1}$ and $\G^{\ge 2m-1,\ell-4}_{a-2,b+2}=\G^{\ge 2m-2,\ell-4}_{a-2,b+2}$ that preserves the major index. The resulting paths,
like those in the original set, start above and end below the reference line. Iterating this argument, we obtain a composition of bijections 
$$\ta_1\circ\si_2\circ\dots\circ\ta_{2m-1}\circ\si_{2m}:\G^{\ge 2m,\ell}_{a,b}\to \G^{\ge 0,\ell-4i}_{a-2m,b+2m}=\G_{a-2m,b+2m}$$
with the property that, if $Q$ is the image of $P$, then
$$\maj(Q)=\maj(P)+(\ell-1)+(\ell-5)+\dots+(\ell-4i+3)=\maj(P)-m(2m-1-\ell).$$
It follows that
$$G_{a,b}^{\ge 2m+1,\ell}(q)=G_{a,b}^{\ge 2m,\ell}(q)=q^{m(2m-1-\ell)}\sum_{Q\in \G_{a-2m,b+2m}}q^{\maj(Q)}=q^{m(2m-1-\ell)}\qbin{a+b}{a-2m},$$
proving Equation~\eqref{eq:0>l>a-b}.

\item {\bf Case III: $0>\ell<a-b$.} 
This case is equivalent to Case I after the first application of $\ta_{2m}$.
Each path must have an even number of crossings, so $\G^{\ge 2m,\ell}_{a,b}=\G^{\ge 2m-1,\ell}_{a,b}$ for all $m\ge1$. 
Since paths in $\G^{\ge 2m-1,\ell}_{a,b}$ end above the reference line, $C_{2m-1}$ is an upward crossing. By Lemma~\ref{lem:maj-tarsir}, $\si_{2m-1}$ is a bijection between $\G^{\ge 2m-1,\ell}_{a,b}$ and $\G^{\ge 2m-1,-\ell+2}_{a-\ell+1,b+\ell-1}=\G^{\ge 2m-2,-\ell+2}_{a-\ell+1,b+\ell-1}$ the changes $\maj$ by $\ell-1$. Continuing as in Case I, we obtain the composition of bijections
$$\si_1\circ\ta_2\circ\dots\circ\si_{2m-3}\circ\ta_{2m-2}\circ\si_{2m-1}:\G^{\ge 2m-1,\ell}_{a,b}\to \G^{\ge 0,-\ell+4i-2}_{a-\ell+2m-1,b+\ell-2m+1}=\G_{a-\ell+2m-1,b+\ell-2m+1}$$
with the property that, if $Q$ is the image of $P$, then
$$\maj(Q)=\maj(P)+(\ell-1)+(\ell-5)+\dots+(\ell-4i+3)=\maj(P)-m(2m-1-\ell).$$
It follows that
$$G_{a,b}^{\ge 2m,\ell}(q)=G_{a,b}^{\ge 2m-1,\ell}(q)=q^{m(2m-1-\ell)}\sum_{Q\in \G_{a-\ell+2m-1,b+\ell-2m+1}}q^{\maj(Q)}=q^{m(2m-1-\ell)}\qbin{a+b}{a+2m-1-\ell},$$
proving Equation~\eqref{eq:0>l<a-b}.

\item {\bf Case IV: $0<\ell>a-b$.} 
As in Case III, $\G^{\ge 2m,\ell}_{a,b}=\G^{\ge 2m-1,\ell}_{a,b}$ for all $m\ge1$. 
For paths in $\G^{\ge 2m-1,\ell}_{a,b}$, the crossing $C_{2m-1}$ is a downward crossing, and $\ta_{2m-1}$ is a bijection between $\G^{\ge 2m-1,\ell}_{a,b}$ and $\G^{\ge 2m-1,-\ell-2}_{a-\ell-1,b+\ell+1}=\G^{\ge 2m-2,-\ell-2}_{a-\ell-1,b+\ell+1}$, which preserves the major index. Continuing as in Case II, we obtain the composition of bijections
$$\ta_1\circ\si_2\circ\dots\circ\ta_{2m-3}\circ\si_{2m-2}\circ\ta_{2m-1}:\G^{\ge 2m-1,\ell}_{a,b}\to \G^{\ge 0,-\ell-4i+2}_{a-\ell-2m+1,b+\ell+2m-1}=\G_{a-\ell-2m+1,b+\ell+2m-1},$$
with the property that, if $Q$ is the image of $P$, then
$$\maj(Q)=\maj(P)-(\ell+3)-(\ell+7)-\dots-(\ell+4i-5)=\maj(P)-(m-1)(2m-1+\ell).$$
It follows that
\begin{multline*}G_{a,b}^{\ge 2m,\ell}(q)=G_{a,b}^{\ge 2m-1,\ell}(q)
=q^{(m-1)(2m-1+\ell)}\sum_{Q\in \G_{a-\ell-2m+1,b+\ell+2m-1}}q^{\maj(Q)}\\
=q^{(m-1)(2m-1+\ell)}\qbin{a+b}{a-2m+1-\ell},\end{multline*}
proving Equation~\eqref{eq:0<l>a-b}.

\item {\bf Case V: $0=\ell<a-b$.} This is the case $a>b$ of Theorem~\ref{thm:xaxis}. Since paths in $\G^{\ge r,0}_{a,b}$ end above the reference line, $C_r$ is a downward crossing if $r$ is even, and an upward crossing if $r$ is odd.

Suppose first that $r$ is even, and write $r=2m$ for some $m\ge1$. The proof in this case is similar to Case~I. By Lemma~\ref{lem:maj-tarsir}, $\ta_{2m}$ is a $\maj$-preserving bijection between $\G^{\ge 2m,0}_{a,b}$ and $\G^{\ge 2m,-2}_{a-1,b+1}=\G^{\ge 2m-1,-2}_{a-1,b+1}$. Paths in this set start and end above the reference line. Continuing as in Case~I with $\ell=0$, we obtain a composition of bijections
$$\si_1\circ\ta_2\circ\dots\circ\si_{2m-1}\circ\ta_{2m}:\G^{\ge 2m,0}_{a,b}\to \G^{\ge 0,4i}_{a+2m,b-2m}=\G_{a+2m,b-2m}$$
that changes $\maj$ by $-m(2m+1)$. It follows that
$$G_{a,b}^{\ge 2m,0}(q)
=q^{m(2m+1)}\qbin{a+b}{a+2m}=q^{\binom{r+1}{2}}\qbin{a+b}{a+r},$$
proving Equation~\eqref{eq:0=l<a-b} for even $r$.

Suppose now that $r$ is odd, and write $r=2m-1$ for some $m\ge1$. By Lemma~\ref{lem:maj-tarsir}, $\si_{2m-1}$ gives a bijection between $\G^{\ge 2m-1,0}_{a,b}$ and $\G^{\ge 2m-1,2}_{a+1,b-1}=\G^{\ge 2m-2,2}_{a+1,b-1}$ that changes $\maj$ by $-1$. Continuing as in Case~III with $\ell=0$, we obtain a composition of bijections
$$\si_1\circ\ta_2\circ\dots\circ\si_{2m-3}\circ\ta_{2m-2}\circ\si_{2m-1}:\G^{\ge 2m-1,0}_{a,b}\to \G^{\ge 0,4i-2}_{a-2m-1,b-2m+1}=\G_{a+2m-1,b-2m+1}$$
that changes $\maj$ by $-m(2m-1)$. It follows that
$$G_{a,b}^{\ge 2m-1,0}(q)
=q^{m(2m-1)}\qbin{a+b}{a+2m-1}=q^{\binom{r+1}{2}}\qbin{a+b}{a+r},$$
proving Equation~\eqref{eq:0=l<a-b} for odd $r$. See Figure~\ref{fig:caseV} for an example.

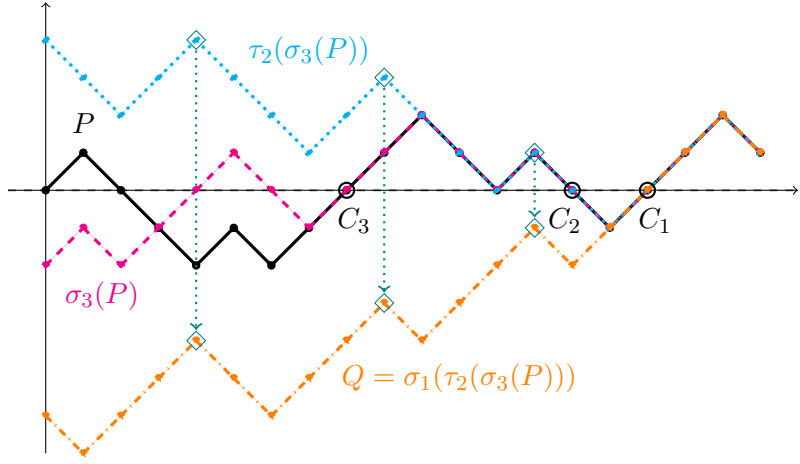
\begin{figure}[htb]
\centering
\begin{tikzpicture}[scale=0.5]
\draw[dashed,thick,gray] (-1,0)--(20,0);
\draw[->] (-1,0)--(20,0);
\draw[->] (0,-7)--(0,5);
\draw[very thick,fill] (0,0) \startud \up\dn\dn\dn\up\dn\up\up \up\up\dn\dn\up\dn \dn\up\up\up\dn;
\draw (1,1.8) node {$P$};
\draw[magenta,dashed,very thick,fill](0,-2)  \startud\up\dn\up\up\up\dn\dn\up \up\up\dn\dn\up\dn \dn\up\up\up\dn;
\draw[magenta] (1.5,-2.8) node {$\si_3(P)$};
\draw[cyan,dotted,very thick,fill](0,4)  \start\dn\dn\up\up\dn\dn\dn\up\up\dn\dn\dn\up\dn  \dn\up\up\up\dn;
\draw[cyan] (7,3.7) node {$\ta_2(\si_3(P))$};
\draw[orange,dash dot,very thick,fill](0,-6)  \start \dn\up\up\up\dn\dn\up\up\up\dn\up\up\up\dn\up\up\up\up\dn;
\draw[orange] (11,-5) node {$Q=\si_1(\ta_2(\si_3(P)))$};
\movepeakvalley{4}{4}
\movepeakvalley{9}{3}
\movepeakvalley{13}{1}
\crossing{8}{0}
\crossing{14}{0}
\crossing{16}{0}
\draw (8.2,-.8) node {$C_3$};
\draw (13.8,-.8) node {$C_2$};
\draw (16.2,-.8) node{$C_1$};
\end{tikzpicture}
\caption{The composition $\si_1\circ\ta_2\circ\si_3$ applied to a path $P\in\G^{\ge 3,0}_{10,9}$. The computation of the maps $\si_3$ and $\ta_2$ is based on the examples in Figure~\ref{fig:tasi-x}.
Note that $\maj(P)=37$ and $\maj(Q)=31=\maj(P)-\binom{r+1}{2}$.
}
\label{fig:caseV}
\end{figure}

\item {\bf Case VI: $0=\ell>a-b$.} This is the case $a>b$ of Theorem~\ref{thm:xaxis}, and it is analogous to Case V. For paths in $\G^{\ge r,0}_{a,b}$, now $C_r$ is an upward crossing if $r$ is even, and a downward crossing if $r$ is odd.

If $r=2m$ for some $m\ge1$, the same argument as in Case~II with $\ell=0$ gives a composition of bijections
$$\ta_1\circ\si_2\circ\dots\circ\ta_{2m-1}\circ\si_{2m}:\G^{\ge 2m,0}_{a,b}\to \G^{\ge 0,-4i}_{a-2m,b+2m}=\G_{a-2m,b+2m}$$
that changes $\maj$ by $-m(2m-1)$.
We deduce that
$$G_{a,b}^{\ge 2m,0}(q)=q^{m(2m-1)}\qbin{a+b}{a-2m}=q^{\binom{r}{2}}\qbin{a+b}{a-r},$$
proving Equation~\eqref{eq:0=l>a-b} for even $r$.

If $r=2m-1$ for some $m\ge1$, the same argument as in Case~IV with $\ell=0$ gives a composition of bijections
$$\ta_1\circ\si_2\circ\dots\circ\ta_{2m-3}\circ\si_{2m-2}\circ\ta_{2m-1}:\G^{\ge 2m-1,0}_{a,b}\to \G^{\ge 0,-4i+2}_{a-2m+1,b+2m-1}=\G_{a-2m+1,b+2m-1}$$
that changes $\maj$ by $-(m-1)(2m-1)$. It follows that
$$G_{a,b}^{\ge 2m-1,0}(q)=q^{(m-1)(2m-1)}\qbin{a+b}{a-2m+1}=q^{\binom{r}{2}}\qbin{a+b}{a-r},$$
proving Equation~\eqref{eq:0=l>a-b} for odd $r$.

\item {\bf Case VII: $0<\ell=a-b$.} 
Denote by  $\G^{\ge r,\ell;D}_{a,b}$ and $\G^{\ge r,\ell;U}_{a,b}$ the subsets of $\G^{\ge r,\ell}_{a,b}$ consisting of paths that end in $D$ and $U$, respectively.
Paths ending in $D$ must have an odd number of crossings, and so $\G^{\ge 2m+1,\ell;D}_{a,b}=\G^{\ge 2m,\ell;D}_{a,b}$ for all $m\ge0$. 
Assuming that $m\ge1$ (the case $m=0$ is solved in Lemma~\ref{lem:qbin}), for paths in this set, $C_{2m}$ is a downward crossing. As in Case~I, and noting that the maps from Lemma~\ref{lem:maj-tarsir} preserve the last step of the path, we obtain a composition of bijections
\begin{equation}\label{eq:si-ta2m}
\si_1\circ\ta_2\circ\dots\circ\si_{2m-1}\circ\ta_{2m}:\G^{\ge 2m+1,\ell;D}_{a,b}=\G^{\ge 2m,\ell;D}_{a,b}\to \G^{\ge 0,\ell+4i;D}_{a+2m,b-2m}=\G^D_{a+2m,b-2m}
\end{equation}
that changes $\maj$ by $-m(2m+1+\ell)$. This map can further be composed with the bijection 
\begin{equation}\label{eq:ta-circ}
\ta:\G^D_{a+2m,b-2m}\to\G^D_{b-2m-1,a+2m+1},
\end{equation}
which preserves $\maj$ by Equation~\eqref{eq:majta2}.

On the other hand, paths ending in $U$ must have an even number of crossings. For such paths, $C_{2m-1}$ is a downward crossing for all $m\ge1$. As in Case~IV, we obtain a bijection
\begin{equation}\label{eq:ta-ta2m-1}
\ta_1\circ\si_2\circ\dots\circ\ta_{2m-3}\circ\si_{2m-2}\circ\ta_{2m-1}:\G^{\ge 2m,\ell;U}_{a,b}=\G^{\ge 2m-1,\ell;U}_{a,b}\to \G^{\ge 0,-\ell-4i+2;U}_{a-\ell-2m+1,b+\ell+2m-1}=\G^U_{b-2m+1,a+2m-1}
\end{equation}
that changes $\maj$ by $-(m-1)(2m-1+\ell)$.
The last equality uses the fact that $\ell=a-b$. This map can be further composed with the bijection 
\begin{equation}\label{eq:si-circ}
\si:\G^U_{b-2m+1,a+2m-1}\to\G^D_{a+2m,b-2m},
\end{equation}
which changes $\maj$ by $-(\ell+4i-1)$, by Equation~\eqref{eq:majsi2}.

To prove the first formula in Equation~\eqref{eq:0<l=a-b}, we construct a bijection $\G^{\ge 2m,\ell}_{a,b}\to\G_{a+2m,b-2m}$ by combining the two bijections
\begin{align*}\si_1\circ\ta_2\circ\dots\circ\si_{2m-1}\circ\ta_{2m}&:\G^{\ge 2m,\ell;D}_{a,b}\to\G^D_{a+2m,b-2m},\\
\si\circ\ta_1\circ\si_2\circ\dots\circ\ta_{2m-1}&:\G^{\ge 2m,\ell;U}_{a,b}\to\G^U_{a+2m,b-2m},
\end{align*}
given by~\eqref{eq:si-ta2m}, and by composing~\eqref{eq:ta-ta2m-1} with~\eqref{eq:si-circ}, respectively.
Both bijections change the major index by $-m(2m+1+\ell)$. It follows that
$$G_{a,b}^{\ge 2m,\ell}(q)=q^{m(2m+1+\ell)}\sum_{Q\in \G_{a+2m,b-2m}}q^{\maj(Q)}=q^{m(2m+1+\ell)}\qbin{a+b}{a+2m}.$$

To prove the second formula in Equation~\eqref{eq:0<l=a-b}, we constuct a bijection $\G^{\ge 2m+1,\ell}_{a,b}\to\G_{b-2m-1,a+2m+1}$ by combining the two bijections
\begin{align*}
\ta_1\circ\si_2\circ\dots\circ\ta_{2m-1}\circ\si_{2m}\circ\ta_{2m+1}&:\G^{\ge 2m+1,\ell;U}_{a,b}\to \G^U_{b-2m-1,a+2m+1},\\
\ta\circ\si_1\circ\ta_2\circ\dots\circ\si_{2m-1}\circ\ta_{2m}&:\G^{\ge 2m+1,\ell;D}_{a,b}\to \G^D_{b-2m-1,a+2m+1},
\end{align*}
given by~\eqref{eq:ta-ta2m-1} with $m+1$ playing the role of $m$, and by composing~\eqref{eq:si-ta2m} with~\eqref{eq:ta-circ}.
Both change the major index by $-m(2m+1+\ell)$. Thus,
$$G_{a,b}^{\ge 2m+1,\ell}(q)=q^{m(2m+1+\ell)}\sum_{Q\in \G_{b-2m-1,a+2m+1}}q^{\maj(Q)}=q^{m(2m+1+\ell)}\qbin{a+b}{a+2m+1}.$$

\item {\bf Case VIII: $0>\ell=a-b$.} This is analogous to Case~VII.
Paths ending in $U$ now must have an odd number of crossings, and for such paths, $C_{2m}$ is an upward crossing for all $m\ge1$. As in Case~II, we have a bijection
$$
\ta_1\circ\si_2\circ\dots\circ\ta_{2m-1}\circ\si_{2m}:\G^{\ge 2m+1,\ell;U}_{a,b}=\G^{\ge 2m,\ell;U}_{a,b}\to 
\G^U_{a-2m,b+2m}
$$
that changes $\maj$ by $-m(2m-1-\ell)$, which can further be composed with the bijection 
$$
\si:\G^U_{a-2m,b+2m}\to\G^U_{b+2m+1,a-2m-1}
$$
that changes $\maj$ by $\ell-4i-1$, by Equation~\eqref{eq:majsi2}.

Paths ending in $D$ have an even number of crossings, and $C_{2m-1}$ is an upward crossing for all $m\ge1$. As in Case~III, we have a bijection
$$\si_1\circ\ta_2\circ\dots\circ\si_{2m-3}\circ\ta_{2m-2}\circ\si_{2m-1}:\G^{\ge 2m,\ell;D}_{a,b}=\G^{\ge 2m-1,\ell;D}_{a,b}\to 
\G^D_{b+2m-1,a-2m+1}
$$
that changes $\maj$ by $-m(2m-1-\ell)$, which can further be composed with the $\maj$-preserving bijection 
$$
\ta:\G^D_{b+2m-1,a-2m+1}\to\G^D_{a-2m,b+2m}.
$$

To prove the first formula in Equation~\eqref{eq:0>l=a-b}, we construct a bijection $\G^{\ge 2m,\ell}_{a,b}\to\G_{a+2m,b-2m}$ by combining the bijections
\begin{align*}
\ta_1\circ\si_2\circ\dots\circ\ta_{2m-1}\circ\si_{2m}&:\G^{\ge 2m,\ell;U}_{a,b}\to \G^U_{a-2m,b+2m},\\
\ta\circ\si_1\circ\ta_2\circ\dots\circ\si_{2m-1}&:\G^{\ge 2m,\ell;D}_{a,b}\to\G^D_{a-2m,b+2m},
\end{align*}
both of which change $\maj$ by $-m(2m-1-\ell)$. 

To prove the second formula in Equation~\eqref{eq:0>l=a-b}, we construct a bijection $\G^{\ge 2m+1,\ell}_{a,b}\to\G_{b+2m+1,a-2m-1}$ by combining the bijections
\begin{align*}
\si_1\circ\ta_2\circ\dots\circ\si_{2m-3}\circ\ta_{2m}\circ\si_{2m+1}&:\G^{\ge 2m+1,\ell;D}_{a,b}\to \G^D_{b+2m+1,a-2m-1},\\
\si\circ\ta_1\circ\si_2\circ\dots\circ\ta_{2m-1}\circ\si_{2m}&:\G^{\ge 2m+1,\ell;U}_{a,b}\to \G^U_{b+2m+1,a-2m-1},
\end{align*}
both of which change $\maj$ by $-(m+1)(2m+1-\ell)$. 

\item {\bf Case IX: $0=\ell=a-b$.}  This is the case $a=b$ of Theorem~\ref{thm:xaxis}.
For paths that end in $D$, removing the last step gives a $\maj$-preserving bijection between 
$\G^{\ge r,0;D}_{a,a}$ and $\G^{\ge r,0}_{a,a-1}$. By Equation~\eqref{eq:0=l<a-b}, proved in Case V, we get
\begin{equation}\label{eq:GDss}
\sum_{P\in \G^{\ge r,0;D}_{a,a}}q^{\maj(P)}=G^{\ge r,0}_{a,a-1}(q)=q^{\binom{r+1}{2}}\qbin{2a-1}{a+r}.
\end{equation}

For paths ending in $U$, the reflection $\rho$ gives a bijection $\rho:\G^{\ge r,0;U}_{a,a}\to\G^{\ge r,0;D}_{a,a}$ such that $\maj(\rho(P))=\maj(P)-a$, by Equation~\eqref{eq:majrho2}. Using Equation~\eqref{eq:GDss}, 
\begin{equation}\label{eq:GUss}
\sum_{P\in \G^{\ge r,0;U}_{a,a}}q^{\maj(P)}=q^a\sum_{P\in \G^{\ge r,0;D}_{a,a}}q^{\maj(P)}=q^{\binom{r+1}{2}+a}\qbin{2a-1}{a+r}.
\end{equation}
Adding Equations~\eqref{eq:GDss} and~\eqref{eq:GUss} proves Equation~\eqref{eq:0=l=a-b}.\qedhere
\end{list1}
\end{proof}

\section{Proofs for pairs of paths crossing each other}
\label{sec:pairs-proofs}

The goal of this section is to prove Theorems~\ref{thm:pairs} and~\ref{thm:pairs_refined}. While it is possible to prove Theorem~\ref{thm:pairs} using certain prefix-swapping bijections, as we will discuss in Section~\ref{sec:noq}, 
proving Theorem~\ref{thm:pairs_refined} requires more sophisticated bijections that keep track of the statistic $\maj$.
We will use these bijections, which rely on the maps $\bt$ and $\bs$ defined in Section~\ref{sec:ingredients}, to prove Theorems~\ref{thm:pairs} and~\ref{thm:pairs_refined} simultaneously.
In the rest of the paper, all paths consist of $N$ and $E$ steps, and the term {\em crossing} always refers to a crossing of two paths. 
Let $A_1,A_2,B_1,B_2,C\in\Z^2$ be arbitrary points, where $A_1=(x_1,y_1)$ and $A_2=(x_2,y_2)$, and let $\vv=(1,-1)$.

We start by stating an immediate consequence of Lemma~\ref{lem:maj}.

\begin{lemma}\label{lem:uptoC}
If $P_\lt\in\P_{A_1\to C}^N$ and $Q_\lt\in\P_{A_2\to C}^E$, then 
$$\maj(\bs(P_\lt))+\maj(\bt(Q_\lt))=\maj(P_\lt)+\maj(Q_\lt)-(x_2-x_1+1).$$
\end{lemma}

\begin{proof}
Suppose that $C=(u,v)$. Then, by Lemma~\ref{lem:maj},
$\maj(\bs(P_\lt))=\maj(P_\lt)-u+x_1$, and 
$\maj(\bt(Q_\lt))=\maj(Q_\lt)+u-x_2-1$. Adding these two equations gives the stated formula.
\end{proof}

Our next task is to define an involution on certain pairs of intersecting paths.
Let $\{\de,\bde\}=\{1,2\}$, and define $\pathsNC{A_1}{B_\de}{A_2}{B_\bde}$ to be the subset of $\paths{A_1}{B_\de}{A_2}{B_\bde}$ consisting of pairs $(P,Q)$ where $C$ is a common point of $P$ and $Q$, the step of $P$ that ends at $C$ is an $N$, and the step of $Q$ that ends at $C$ is an $E$.

\begin{definition}\label{def:phiC}
For $(P,Q)\in\pathsNC{A_1}{B_\de}{A_2}{B_\bde}$, write $P=P_\lt P_\rt$ and $Q=Q_\lt Q_\rt$ by splitting both paths at $C$,
and let $$\phiC(P,Q)=\left(\bt(Q_\lt)P_\rt,\bs(P_\lt)Q_\rt\right).$$
\end{definition}

See Figure~\ref{fig:phiC} for an example.

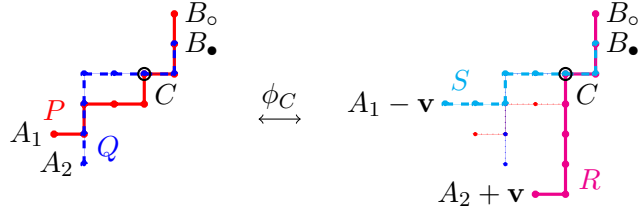
\begin{figure}[htb]
\centering
\begin{tikzpicture}[scale=0.4]
    \draw[red,very thick,fill](0,1) \start\E\N\E\E\N\E\N\N;
    \draw[red] (0,1.8) node {$P$};
    \draw[blue,dashed,very thick,fill](1,0) \start\N\N\N\E\E\E\N;
    \draw[blue] (1.8,0.5) node {$Q$};
	\draw (0,1) node[left] {$A_1$};
	\draw (1,0) node[left] {$A_2$};     
	\draw (4,5) node[right] {$B_\de$};
	\draw (4,4) node[right] {$B_\bde$};    
\crossing{3}{3};
	\draw (3,3) node[below right] {$C$}; 
	\draw[<->] (6.8,1.5)-- node[above]{$\phiC$} (8.2,1.5);

\begin{scope}[shift={(14,0)}] 
    \draw[red,very thin,dotted,fill](0,1) \start\E\N\E\E\N;
    \draw[blue,very thin,dotted,fill](1,0) \start\N\N\N\E\E;
    \draw[magenta,very thick,fill](2,-1) \start\E\N\N\N\N\E\N\N;
    \draw[magenta] (3.8,-.5) node {$R$};
    \draw[cyan,dashed,very thick,fill](-1,2) \start\E\E\N\E\E\E\N;
    \draw[cyan] (-.5,2.8) node {$S$};
	\draw (-1,2) node[left] {$A_1-\vv$};
	\draw (2,-1) node[left]{$A_2+\vv$};     
	\draw (4,5) node[right] {$B_\de$};
	\draw (4,4) node[right] {$B_\bde$};   
\crossing{3}{3};
	\draw (3,3) node[below right] {$C$}; 
\end{scope}
\end{tikzpicture}
\caption{An example of the bijection $\phiC$.}
\label{fig:phiC}
\end{figure}

\begin{lemma}\label{lem:phiC}
The map from Definition~\ref{def:phiC} is a bijection
$$\phiC:\pathsNC{A_1}{B_\de}{A_2}{B_\bde}\to \pathsNC{A_2+\vv}{B_\de}{A_1-\vv}{B_\bde}.$$
Additionally, if condition~\eqref{condition} holds and $\phiC(P,Q)=(R,S)$, then
\begin{equation}\label{eq:majP'Q'C}
\maj(R)+\maj(S)=\maj(P)+\maj(Q)-(x_2-x_1+1).
\end{equation}
\end{lemma}

\begin{proof}
For $(P,Q)\in\pathsNC{A_1}{B_\de}{A_2}{B_\bde}$, the decomposition in Definition~\ref{def:phiC} yields paths $P_\lt\in\P_{A_1\to C}^N$ and $Q_\lt\in\P_{A_2\to C}^E$.
Thus, by Lemma~\ref{lem:bsbt}, $\bt(Q_\lt)\in\P_{A_2+\vv\to C}^N$ and $\bs(P_\lt)\in\P_{A_1-\vv\to C}^E$.
It follows that
$$R=\bt(Q_\lt)P_\rt\in \P_{A_2+\vv\to B_\de} \quad\text{and}\quad  S=\bs(P_\lt)Q_\rt\in \P_{A_1-\vv\to B_\bde},$$
that $C$ is a common point of $R$ and $S$, and that the step of $R$ (resp.\ $S$) that ends $C$ is an $N$ (resp.\ $E$). Thus, $(R,S)\in\pathsNC{A_2+\vv}{B_\de}{A_1-\vv}{B_\bde}$.

Next we show that the inverse of $\phiC$ is given by the same map $\phiC$ on the appropriate domain $\pathsNC{A_2+\vv}{B_\de}{A_1-\vv}{B_\bde}$. Indeed, for $(R,S)$ as above, we have $\phiC(R,S)=(P,Q)$, using the fact that $\bt(\bs(P_\lt))=P_\lt$ and $\bs(\bt(Q_\lt))=Q_\lt$ by Lemma~\ref{lem:bsbt}.

Finally, let us compare $\maj(R)+\maj(S)$ to $\maj(P)+\maj(Q)$, assuming that condition~\eqref{condition} holds.
The contributions from valleys occurring at or after $C$ are the same for both sums, since both $R$ and $P$ end with $NP_\rt$, both $S$ and $Q$ end with $EQ_\rt$, and each of the four starting points $A_1$, $A_2$, $A_1-\vv$, $A_2+\vv$ has the same coordinate sum because of condition~\eqref{condition}. On the other hand, the contributions from valleys occurring before $C$
change according to Lemma~\ref{lem:uptoC}:
$$\maj(\bs(P_\lt))+\maj(\bt(Q_\lt))=\maj(P_\lt)+\maj(Q_\lt)-(x_2-x_1+1).$$
Equation~\eqref{eq:majP'Q'C} now follows.
\end{proof}

Let $r\ge0$. Given a pair $(P,Q)\in\pathsP{A_1}{B_\de}{A_2}{B_\bde}{r}$, let $C_j=C_j(P,Q)$ denote the $j$th crossing of $P$ and $Q$ starting from the right, for $1\le j\le r$. In the special case that $B_\de=B_\bde$ (call this point $B$),
we additionally define $C_0=C_0(P,Q)$ to be the last common vertex of $P$ and $Q$ when reading these paths backwards from $B$; in other words, $C_0$ is such that the maximal common suffix of $P$ and $Q$ has endpoints $C_0$ and $B$. 

For $r\ge0$, define $\pathsN{A_1}{B_\de}{A_2}{B_\bde}{r}$ (resp.\ $\pathsE{A_1}{B_\de}{A_2}{B_\bde}{r}$)
to be the subset of $\pathsP{A_1}{B_\de}{A_2}{B_\bde}{r}$ consisting of pairs $(P,Q)$ for which $C_r$ is defined (this condition is only meaningful when $r=0$), and such that the step of $P$ that ends at $C_r$ is an $N$ (resp.\ an $E$), and the step of $Q$ that ends at $C_r$ is an $E$ (resp.\ an $N$). 
Let $\varsigma$ be the involution on pairs of paths defined by 
\begin{equation}\label{eq:swap}
\varsigma(P,Q)=(Q,P).
\end{equation}
Note that $\varsigma$ restricts to a bijection between $\pathsN{A_1}{B_\de}{A_2}{B_\bde}{r}$ and $\pathsE{A_2}{B_\bde}{A_1}{B_\de}{r}$, and that it trivially preserves the total major index.

Suppose that $r\ge0$ if $B_\de=B_\bde$ and $A_1\prec A_2$, and that $r\ge1$ otherwise. Then we have the decomposition
\begin{equation}\label{eq:decomp0}
\pathsP{A_1}{B_\de}{A_2}{B_\bde}{r}=\pathsN{A_1}{B_\de}{A_2}{B_\bde}{r}\sqcup\pathsE{A_1}{B_\de}{A_2}{B_\bde}{r},
\end{equation}
where $\sqcup$ denotes disjoint union. Note that, if $B_\de\prec B_\bde$, then the first set in the right-hand side of Equation~\eqref{eq:decomp0} is empty for even $r$, and the second one is empty for odd $r$; if $B_\bde\prec B_\de$, a similar statement holds with the parities reversed.
Now we are ready to define~$\bij_r$.

\begin{definition}\label{def:bijr}
Let $r\ge0$ if $B_\de=B_\bde$ and $A_1\prec A_2$, and let $r\ge1$ otherwise.
For $(P,Q)\in\pathsN{A_1}{B_\de}{A_2}{B_\bde}{r}$, let $C=C_r(P,Q)$, and define 
$$\bij_r(P,Q)=\varsigma(\phiC(P,Q)).$$
\end{definition}

See the examples in Figure~\ref{fig:Bij}.

\begin{figure}[htb]
\centering
\begin{tikzpicture}[scale=0.4]
    \draw[red,very thick,fill](0,1) \start\N\N\E\E\E\N\N\N\N\E\E\N\E\N\E\E;
    \draw[red] (0,2.5) node[left] {$P$};
    \draw[blue,dashed,very thick,fill](1,0) \start\N\N\N\N\E\N\E\E\N\N\E\E\N\N\N\N\E;
    \draw[blue] (1,1.5) node[right] {$Q$};
\crossing{3}{5};
	\draw (3,5) node[below right] {$C_2$};
\crossing{6}{8};
	\draw (6,8) node[right] {$C_1$};
	\draw (0,1) node[left] {\small $A_1$};
	\draw (1,0) node[left] {\small $A_2$};     
	\draw (7,11) node[right] {\small $B_1$};
	\draw (8,9) node[right] {\small $B_2$};     
	\draw[->] (8.5,5.5)-- node[above]{$\bij_2$} (9.5,5.5);

\begin{scope}[shift={(12,0)}] 
    \draw[red,very thin,dotted,fill](0,1) \start\N\N\E\E\E\N\N;
    \draw[magenta,very thick,fill](-1,2) \start\N\E\N\N\E\E\E \E\N\N\E\E\N\N\N\N\E ;
    \draw[magenta] (0,4.5) node[left] {$P'$};
    \draw[blue,dotted,very thin,fill](1,0) \start\N\N\N\N\E\N\E;
    \draw[cyan,dashed,very thick,fill](2,-1) \start\N\N\N\N\N\E\N \N\N\E\E\N\E\N\E\E;
    \draw[cyan] (2,1.5) node[right] {$Q'$};
\crossing{3}{5};
	\draw (3,5) node[below right] {$C_2$};
\crossing{6}{8};
	\draw (6,8) node[right] {$C_1$};
	\draw (-1,2) node[left] {\small $A_1-\vv$};
	\draw (2,-1) node[left] {\small $A_2+\vv$};     
	\draw (7,11) node[right] {\small $B_1$};
	\draw (8,9) node[right] {\small $B_2$};     
	\draw[->] (8.5,5.5)-- node[above]{$\bij_1$} (9.5,5.5);
\end{scope}

\begin{scope}[shift={(26,0)}] 
    \draw[magenta,very thin,dotted,fill](-1,2) \start\N\E\N\N\E\E\E\E\N\N\E\E\N;
    \draw[orange,very thick,fill](-2,3) \start \E\N\N\E\N\N\E\E\E\E\N\E\E \N\E\E;
    \draw[orange] (-1,4.5) node[left] {$\hat P=P''$};
    \draw[cyan,dotted,very thin,fill](2,-1) \start\N\N\N\N\N\E\N\N\N\E\E\N\E ;
    \draw[violet,dashed,very thick,fill](3,-2) \start \N\N\N\N\N\N\E\E\N\N\N\E\N \N\N\N\E;
    \draw[violet] (3,1.5) node[right] {$\hat Q=Q''$};
\crossing{6}{8};
	\draw (6,8) node[right] {$C_1$};
	\draw (-2,3) node[left] {\small $A_1-2\vv$};
	\draw (3,-2) node[left] {\small $A_2+2\vv$};     
	\draw (7,11) node[right] {\small $B_1$};
	\draw (8,9) node[right] {\small $B_2$};     
\end{scope}

\end{tikzpicture}
\caption{An example of the bijection $\Bij_2$ in Case~1, as a composition $(P,Q)\stackrel{\bij_2}{\mapsto}(P',Q')\stackrel{\bij_1}{\mapsto}(P'',Q'')$. Note that condition~\eqref{condition} holds, and that $\maj(P)+\maj(Q)=(5+11+13)+(5+8+12)=54$, $\maj(P')+\maj(Q')=(2+8+12)+(6+11+13)=52$, and $\maj(P'')+\maj(Q'')=(1+4+10+13)+(8+12)=48$. Thus, $\bij_2$ decreases the total major index by $x_2-x_1+1=2$, and $\bij_1$ decreases it by $(x_2+1)-(x_1-1)+1=4$.
}
\label{fig:Bij}
\end{figure}
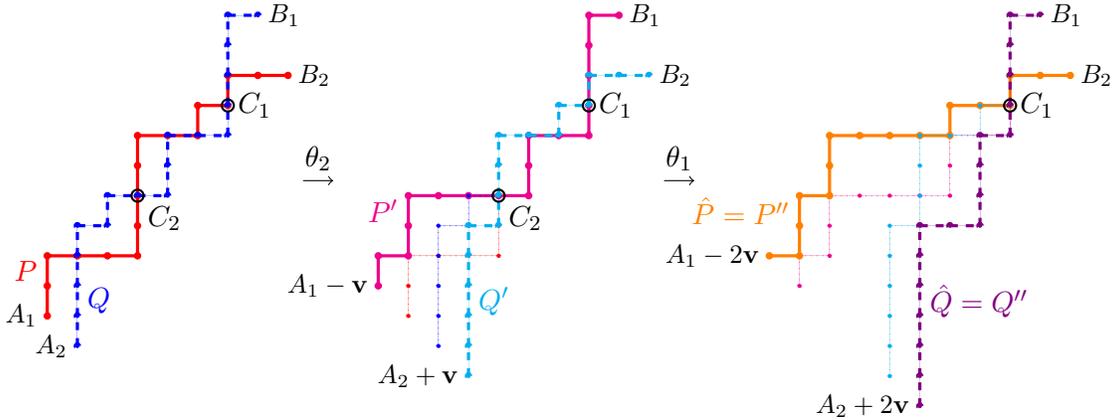

\begin{lemma}\label{lem:bijr}
The map from Definition~\ref{def:bijr} is a bijection
$$\bij_r:\pathsN{A_1}{B_\de}{A_2}{B_\bde}{r}\to\pathsE{A_1-\vv}{B_\bde}{A_2+\vv}{B_\de}{r}.$$
Additionally, if condition~\eqref{condition} holds and $\bij_r(P,Q)=(P',Q')$, then
\begin{equation}\label{eq:majP'Q'}
\maj(P')+\maj(Q')=\maj(P)+\maj(Q)-(x_2-x_1+1).
\end{equation}
\end{lemma}

\begin{proof}
Given a pair $(P,Q)\in\pathsN{A_1}{B_\de}{A_2}{B_\bde}{r}$, applying $\phiC$ with $C=C_r(P,Q)$ preserves the suffixes of the paths after $C$. In particular, it preserves the rightmost $r$ crossings $C_1,C_2,\dots,C_r$, and also $C_0$ in the case $r=0$.
Combined with Lemma~\ref{lem:phiC}, this implies that $\phiC$ induces a bijection from $\pathsN{A_1}{B_\de}{A_2}{B_\bde}{r}$ to $\pathsN{A_2+\vv}{B_\de}{A_1-\vv}{B_\bde}{r}$. Composing with $\varsigma$ yields a bijection to $\pathsE{A_1-\vv}{B_\bde}{A_2+\vv}{B_\de}{r}$.
 Equation~\eqref{eq:majP'Q'} follows trivially from Equation~\eqref{eq:majP'Q'C}.
\end{proof}

\begin{proof}[Proof of Theorems~\ref{thm:pairs} and~\ref{thm:pairs_refined}]
We separate cases according to which endpoints of the paths coincide. We will prove both theorems in parallel, requiring condition~\eqref{condition} only for the refined formulas in Theorem~\ref{thm:pairs_refined} that keep track of $\maj$. We use the notation $C_j=C_j(P,Q)$ defined above. The right equality in Equations~\eqref{eq:switched-noq} and~\eqref{eq:switched} for $m=0$, as well as Equations \eqref{eq:A1=A2-noq}--\eqref{eq:A1=A2,B1=B2-noq} and \eqref{eq:A1=A2}--\eqref{eq:A1=A2,B1=B2} for $r=0$ are implied by Equations~\eqref{eq:m=0noq} and~\eqref{eq:m=0}, so we will assume that $m\ge1$ and $r\ge1$ when proving these.

\begin{list1}
\item {\bf Case 1:} endpoints $A_1\prec A_2$ and $B_1\prec B_2$.
If $P\in\P_{A_1\to B_2}$ and $Q\in\P_{A_2\to B_1}$, then $\cro(P,Q)$ must be odd, because of the relative position of the endpoints of the two paths. Additionally, the step of $P$ that ends at $C_j$ is an $N$ for even $j$, and it is an $E$ for odd $j$. 
It follows that, for $m\ge1$,
\begin{equation}\label{eq:odd=even}
\pathsE{A_1}{B_2}{A_2}{B_1}{2m+1}=\pathsP{A_1}{B_2}{A_2}{B_1}{2m+1}=\pathsP{A_1}{B_2}{A_2}{B_1}{2m}=\pathsN{A_1}{B_2}{A_2}{B_1}{2m},
\end{equation}
and that the first two equalities also hold for $m=0$. This implies the left equality in Equations~\eqref{eq:switched-noq} and~\eqref{eq:switched} for all $m\ge0$.

Similarly, if $P\in\P_{A_1\to B_1}$ and $Q\in\P_{A_2\to B_2}$, then $\cro(P,Q)$ must be even. Now the step of $P$ that ends at $C_j$ is an $N$ for odd $j$, and it is an $E$ for even $j$.
Thus, for $m\ge1$,
\begin{equation}\label{eq:even=odd}
\pathsE{A_1}{B_1}{A_2}{B_2}{2m}=\pathsP{A_1}{B_1}{A_2}{B_2}{2m}=\pathsP{A_1}{B_1}{A_2}{B_2}{2m-1}=\pathsN{A_1}{B_1}{A_2}{B_2}{2m-1},
\end{equation}
proving the left equality in Equations~\eqref{eq:same-noq} and~\eqref{eq:same}.

To prove the right equalities in these four equations, let us assume that $m\ge1$. Setting $r=2m$ and $r=2m-1$ in Lemma~\ref{lem:bijr}, respectively, and using Equations~\eqref{eq:odd=even} and~\eqref{eq:even=odd}, which also hold for the initial points $A_1-\vv\prec A_2+\vv$, we get bijections
$$\begin{array}{rcl}
\pathsP{A_1}{B_2}{A_2}{B_1}{2m}=\pathsN{A_1}{B_2}{A_2}{B_1}{2m}&\stackrel{\bij_{2m}}{\to}&
\pathsE{A_1-\vv}{B_1}{A_2+\vv}{B_2}{2m}=\pathsP{A_1-\vv}{B_1}{A_2+\vv}{B_2}{2m-1},\\
\pathsP{A_1}{B_1}{A_2}{B_2}{2m-1}=\pathsN{A_1}{B_1}{A_2}{B_2}{2m-1}&\stackrel{\bij_{2m-1}}{\to}&
\pathsE{A_1-\vv}{B_2}{A_2+\vv}{B_1}{2m-1}=\pathsP{A_1-\vv}{B_2}{A_2+\vv}{B_1}{2m-2}
\end{array}.$$

Thus, the compositions $\Bij_{r}=\bij_{1}\circ\bij_{2}\circ\dots\circ\bij_{r}$ give bijections
\begin{align*}
\Bij_{2m}&:\pathsP{A_1}{B_2}{A_2}{B_1}{2m}\to\paths{A_1-2m\vv}{B_2}{A_2+2m\vv}{B_1},\\
\Bij_{2m-1}&:\pathsP{A_1}{B_1}{A_2}{B_2}{2m-1}\to\paths{A_1-(2m-1)\vv}{B_2}{A_2+(2m-1)\vv}{B_1}
\end{align*}
for all $m\ge1$. An example of the bijection $\Bij_2$ is given in Figure~\ref{fig:Bij}.
Equations~\eqref{eq:switched-noq} and~\eqref{eq:same-noq} immediately follow using Equation~\eqref{eq:m=0noq} and the fact that $A_1-r\vv=(x_1-r,y_1+r)$ and $A_2+r\vv=(x_2+r,y_2-r)$.

Let us now assume that condition~\eqref{condition} holds. This implies that the sum of the two coordinates is the same for all the points of the form $A_1-j\vv$ and $A_2+j\vv$.
If we let $\Bij_{r}(P,Q)=(\hat P,\hat Q)$, then repeated applications of Lemma~\ref{lem:bijr} give 
\begin{align}\nonumber
\maj(\hat P)+\maj(\hat Q)&=\maj(P)+\maj(Q)-(x_2-x_1+1)-(x_2-x_1+3)-\dots-(x_2-x_1+2r-1)\\
&=\maj(P)+\maj(Q)-r(r+x_2-x_1).
\label{eq:maj-Bijr}
\end{align}
For $r=2m$, this property of the bijection $\Bij_{2m}$, together with Equation~\eqref{eq:m=0}, implies that
\begin{align*}H^{\ge 2m}_{A_1\to B_2,A_2\to B_1}(q)&=q^{2m(2m+x_2-x_1)}H^{\ge 0}_{A_1-2m\vv\to B_2,A_2+2m\vv\to B_1}(q)\\
&=q^{2m(2m+x_2-x_1)}\qbin{u_2-x_1+v_2-y_1}{u_2-x_1+2m}\qbin{u_1-x_2+v_1-y_2}{u_1-x_2-2m}=f_{2m,A_1,A_2,B_2,B_1}(q),\end{align*}
proving Equation~\eqref{eq:switched}. A similar argument for $r=2m-1$ proves Equation~\eqref{eq:same}. 

\item {\bf Case 2:} endpoints $A$ and $B_1\prec B_2$. Assume that $r\ge1$, and let $(P,Q)\in\pathsP{A}{B_1}{A}{B_2}{r}$. The relative position of $B_1$ and $B_2$ forces the step of $P$ that ends at $C_r$ to be an $N$ if $r$ is odd, and an $E$ if $r$ is even. Thus, writing $r=2m+1$ or $r=2m$ accordingly, we have
$$\pathsP{A}{B_1}{A}{B_2}{2m+1}=\pathsN{A}{B_1}{A}{B_2}{2m+1}\quad \text{and} \quad \pathsP{A}{B_1}{A}{B_2}{2m}=\pathsE{A}{B_1}{A}{B_2}{2m}.$$
By Lemma~\ref{lem:bijr} and Equation~\eqref{eq:odd=even} with initial points $A-\vv\prec A+\vv$, in the odd case we get a bijection
$$\pathsP{A}{B_1}{A}{B_2}{2m+1}=\pathsN{A}{B_1}{A}{B_2}{2m+1}\stackrel{\bij_{2m+1}}{\to}\pathsE{A-\vv}{B_2}{A+\vv}{B_1}{2m+1}=\pathsP{A-\vv}{B_2}{A+\vv}{B_1}{2m}.$$
In the even case, we first apply the swap $\varsigma$ from Equation~\eqref{eq:swap}, which gives a bijection 
$$\pathsP{A}{B_1}{A}{B_2}{2m}=\pathsE{A}{B_1}{A}{B_2}{2m}\stackrel{\varsigma}{\to}\pathsN{A}{B_2}{A}{B_1}{2m},$$ 
and then use Lemma~\ref{lem:bijr} and Equation~\eqref{eq:even=odd} to get a bijection
$$\pathsN{A}{B_2}{A}{B_1}{2m}\stackrel{\bij_{2m}}{\to}\pathsE{A-\vv}{B_1}{A+\vv}{B_2}{2m}=\pathsP{A-\vv}{B_1}{A+\vv}{B_2}{2m-1}.$$
The images of the above maps $\bij_{2m+1}$ and $\bij_{2m}$ consist of pairs of paths where neither the starting nor the final points coincide, so we can apply to these sets the bijections $\Bij_{2m}$ and $\Bij_{2m-1}$ as in Case~1, respectively.

For $r=2m+1$, the composition $\Bij_{2m}\circ\bij_{2m+1}$ yields a bijection
$$\Bij_{2m+1}:\pathsP{A}{B_1}{A}{B_2}{2m+1}\to\paths{A-(2m+1)\vv}{B_2}{A+(2m+1)\vv}{B_1},$$
and for $r=2m$, the composition $\Bij_{2m-1}\circ\bij_{2m}\circ\varsigma$ yields a bijection
$$\Bij_{2m}\circ\varsigma:\pathsP{A}{B_1}{A}{B_2}{2m}\to\paths{A-2m\vv}{B_2}{A +2m\vv}{B_1}.$$
These two bijections, together with Equation~\eqref{eq:m=0noq}, prove Equation~\eqref{eq:A1=A2-noq} for both odd and even $r$.

Additionally, condition~\eqref{condition} holds for the initial points in all the above sets, since all points of the form $A+ j\vv$ for $j\in\Z$ have the same coordinate sum. Using Lemma~\ref{lem:bijr}, the same calculation from Equation~\eqref{eq:maj-Bijr} shows that
$\Bij_{2m+1}$ and $\Bij_{2m}\circ\varsigma$ change the total major index by $-r^2$. Hence, by Equation~\eqref{eq:m=0}, these bijections prove Equation~\eqref{eq:A1=A2}. 

\item {\bf Case 3:} endpoints $A_1\prec A_2$ and $B$. Given $P\in\P_{A_1\to B}$ and $Q\in\P_{A_2\to B}$, the relative position of $A_1$ and $A_2$ implies that, if $P$ arrives at $C_j$ (where $j\ge0$) with an $N$ step, then there must be another crossing to the left of $C_j$.
It follows that 
\begin{equation}\label{eq:j=j+1}
\pathsN{A_1}{B}{A_2}{B}{j}=\pathsE{A_1}{B}{A_2}{B}{j+1}
\end{equation} 
for all $j\ge0$.

To prove Equations~\eqref{eq:B1=B2-noq} and~\eqref{eq:B1=B2} for $r\ge1$, we start with the decomposition~\eqref{eq:decomp0} for $B_\de=B_\bde=B$. 
For the set $\pathsN{A_1}{B}{A_2}{B}{r}$, Lemma~\ref{lem:bijr} gives a bijection
$$\pathsN{A_1}{B}{A_2}{B}{r}\stackrel{\bij_r}{\to}\pathsE{A_1-\vv}{B}{A_2+\vv}{B}{r}=\pathsN{A_1-\vv}{B}{A_2+\vv}{B}{r-1},$$
using Equation~\eqref{eq:j=j+1} with initial points $A_1-\vv\prec A_2+\vv$. Thus, the composition $\Bij_r=\bij_1\circ\bij_2\circ\dots\circ\bij_r$ gives a bijection
\begin{equation}\label{eq:BijrN}
\Bij_r:\pathsN{A_1}{B}{A_2}{B}{r}\to \pathsN{A_1-r\vv}{B}{A_2+r\vv}{B}{0}.
\end{equation}

On the other hand, for the set $\pathsE{A_1}{B}{A_2}{B}{r}$, Equation~\eqref{eq:j=j+1} and Lemma~\ref{lem:bijr} give a bijection
$$\pathsE{A_1}{B}{A_2}{B}{r}=\pathsN{A_1}{B}{A_2}{B}{r-1}\stackrel{\bij_{r-1}}{\to} \pathsE{A_1-\vv}{B}{A_2+\vv}{B}{r-1}.$$
Thus, the composition $\Bij_{r-1}=\bij_1\circ\bij_2\circ\dots\circ\bij_{r-1}$ gives a bijection
$$\Bij_{r-1}:\pathsE{A_1}{B}{A_2}{B}{r}\to \pathsE{A_1-(r-1)\vv}{B}{A_2+(r-1)\vv}{B}{1}.$$
The right-hand side equals $\pathsN{A_1-(r-1)\vv}{B}{A_2+(r-1)\vv}{B}{0}$ by Equation~\eqref{eq:j=j+1}, and this set is in bijection with $\pathsE{A_1-r\vv}{B}{A_2+r\vv}{B}{0}$ by Lemma~\ref{lem:bijr} with $r=0$. The composition yields a bijection 
\begin{equation}\label{eq:BijrE}
\bij_0\circ\Bij_{r-1}:\pathsE{A_1}{B}{A_2}{B}{r}\to \pathsE{A_1-r\vv}{B}{A_2+r\vv}{B}{0}.
\end{equation}

Combining~\eqref{eq:BijrN} and~\eqref{eq:BijrE}, and using the decomposition~\eqref{eq:decomp0} on the range and on the domain, we get a bijection from $\pathsP{A_1}{B}{A_2}{B}{r}$ to 
$\paths{A_1-r\vv}{B}{A_2+r\vv}{B}$, which proves  Equation~\eqref{eq:B1=B2-noq}.

If condition~\eqref{condition} is satisfied, then Lemma~\ref{lem:bijr} implies that Equation~\eqref{eq:maj-Bijr} holds when $(\hat P,\hat Q)$ is the image of $(P,Q)$ by either of the maps~\eqref{eq:BijrN} or~\eqref{eq:BijrE}. Thus, the total major index changes by $-r(r+x_2-x_1)$ in either case, proving Equation~\eqref{eq:B1=B2}. 

\item {\bf Case 4:} endpoints $A$ and $B$. Assume that $r\ge1$. First observe that the map $\varsigma$ from Equation~\eqref{eq:swap} gives a trivial bijection between $\pathsN{A}{B}{A}{B}{r}$ and $\pathsE{A}{B}{A}{B}{r}$ which preserves the total major index.
Using the decomposition~\eqref{eq:decomp0}, it follows that
\begin{equation}\label{eq:HN}
H^{\ge r}_{A\to B,A\to B}(q)=2\sum_{(P,Q)\in\pathsN{A}{B}{A}{B}{r}} q^{\maj(P)+\maj(Q)}.
\end{equation}

Lemma~\ref{lem:bijr} gives a bijection
$$\pathsN{A}{B}{A}{B}{r}\stackrel{\bij_r}{\to}\pathsE{A-\vv}{B}{A+\vv}{B}{r}=\pathsN{A-\vv}{B}{A+\vv}{B}{r-1},$$
using Equation~\eqref{eq:j=j+1} with intial points $A-\vv\prec A+\vv$. Thus, the composition $\Bij_r=\bij_1\circ\bij_2\circ\dots\circ\bij_r$ gives a bijection 
\begin{equation}\label{eq:BijrNA}
\Bij_r:\pathsN{A}{B}{A}{B}{r}\to \pathsN{A-r\vv}{B}{A+r\vv}{B}{0}
\end{equation}
that changes the total major index by $-r^2$. Composing with $\bij_0$,
we get a bijection 
\begin{equation}\label{eq:bij0BijrN}
\bij_0\circ\Bij_r:\pathsN{A}{B}{A}{B}{r}\to \pathsE{A-(r+1)\vv}{B}{A+(r+1)\vv}{B}{0}
\end{equation}
that changes the total major index by $-(r+1)^2$.

Combining~\eqref{eq:bij0BijrN} and~\eqref{eq:BijrNA}, with $r+1$ playing the role of $r$ in the latter, and 
using the decomposition~\eqref{eq:decomp0} with $r=0$, initial points $A-(r+1)\vv$ and $A+(r+1)\vv$, and final points $B$ for both paths, we obtain
\begin{multline*}
\sum_{(P,Q)\in\pathsN{A}{B}{A}{B}{r}} q^{\maj(P)+\maj(Q)}+\sum_{(P,Q)\in\pathsN{A}{B}{A}{B}{r+1}} q^{\maj(P)+\maj(Q)}\\
=q^{(r+1)^2} H^{\ge0}_{A-(r+1)\vv\to B, A+(r+1)\vv\to B}=q^{(r+1)^2} \qbin{u-x+v-y}{u-x+r+1}\qbin{u-x+v-y}{u-x-r-1},
\end{multline*}
where the last equality uses Equation~\eqref{eq:m=0}. By Equations~\eqref{eq:HN} and~\eqref{eq:fr}, this is equivalent to
\begin{equation}\label{eq:HrHr+1}
H^{\ge r}_{A\to B,A\to B}(q)+H^{\ge r+1}_{A\to B,A\to B}(q)=2f_{r+1,A,A,B,B}(q).
\end{equation}
Solving for $H^{\ge r}_{A\to B,A\to B}(q)$ and iterating, we obtain
$$H^{\ge r}_{A\to B,A\to B}(q)=2\left(f_{r+1,A,A,B,B}(q)-f_{r+2,A,A,B,B}(q)+f_{r+3,A,A,B,B}(q)-\cdots\right),$$
which proves Equation~\eqref{eq:A1=A2,B1=B2}, and hence Equation~\eqref{eq:A1=A2,B1=B2-noq} as well by setting $q=1$.\qedhere
\end{list1}
\end{proof}

It is also possible to obtain an alternative expression for Equations~\eqref{eq:A1=A2,B1=B2} and~\eqref{eq:A1=A2,B1=B2-noq}  by iterating the recurrence~\eqref{eq:HrHr+1} in the other direction, by decreasing $r$ instead. When the iteration reaches $r=0$, Equation~\eqref{eq:HN} no longer holds, and~\eqref{eq:decomp0} must be replaced by the decomposition
$$\paths{A}{B}{A}{B}=\{(P,P):P\in\P_{A\to B}\}\sqcup \pathsN{A}{B}{A}{B}{0}\sqcup\pathsE{A}{B}{A}{B}{0}.$$
Indeed, for pairs $(P,Q)$ in the left-hand where $P\neq Q$, the paths $P$ and $Q$ must arrive at $C_0$ with different steps. 
Enumerating each set in the decomposition by total major index, using Lemma~\ref{lem:qbin},
and noting that the last two sets are in bijection with each other via the swap $\varsigma$, we obtain the identity
$$\qbin{u-x+v-y}{u-x}^2=\qbinsq{u-x+v-y}{u-x}+2\sum_{(P,Q)\in\pathsN{A}{B}{A}{B}{0}} q^{\maj(P)+\maj(Q)}.$$

Iterating Equation~\eqref{eq:HrHr+1} by decreasing $r$, and using the expression
$$2\sum_{(P,Q)\in\pathsN{A}{B}{A}{B}{0}} q^{\maj(P)+\maj(Q)}=\qbin{u-x+v-y}{u-x}^2-\qbinsq{u-x+v-y}{u-x}$$
in place of $H^{\ge 0}_{A\to B,A\to B}(q)$, we get the alternative formula
\begin{equation}\label{eq:A1=A2,B1=B2-alt}
H^{\ge r}_{A\to B,A\to B}(q)=2\sum_{j=0}^{r-1}(-1)^{j}f_{r-j,A,A,B,B}(q)+(-1)^r\left(\qbin{u-x+v-y}{u-x}^2-\qbinsq{u-x+v-y}{u-x}\right)
\end{equation}
for $r\ge1$. Setting $q=1$, we obtain
\begin{align}\label{eq:A1=A2,B1=B2-noq-alt}
\left|\P^{\ge r}_{A\to B,A\to B}\right|&=2\sum_{j=0}^{r-1}(-1)^{j}\binom{u-x+v-y}{u-x+r-j}\binom{u-x+v-y}{u-x-r+j}\\
&\quad +(-1)^r\left(\binom{u-x+v-y}{u-x}^2-\binom{u-x+v-y}{u-x}\right). \qedhere
\nonumber
\end{align}

\section{Connections to non-intersecting paths}\label{sec:connections}

\subsection{Enumerating tuples of non-intersecting paths by major index}\label{sec:non-int}

In this section we apply some of the above constructions to give an alternative proof of Krattenthaler's beautiful refinement \cite[Thm.\ 2]{Krat-nonint} by total major index of the Lindstr\"om--Gessel--Viennot determinantal formula enumerating $k$-tuples of non-intersecting lattice paths~\cite{Lin,GV}.
For a tuple $\bP=(P_1,P_2,\dots,P_k)$ of paths with $N$ and $E$ steps, define $\maj(\bP)=\sum_{i=1}^k \maj(P_i)$. We say that $\bP$ is intersecting if some point in $\Z^2$ is shared by more than one path in $\bP$, and that it is non-intersecting otherwise, namely, if all the paths are disjoint.

\begin{theorem}[\cite{Krat-nonint}]\label{thm:non-int}
Let $I_i=(x_i,y_i)$ and $F_i=(u_i,v_i)$ be points in $\Z^2$ for $1\le i\le k$, with $I_1\prec I_2\prec \dots \prec I_k$ and $F_1\prec F_2\prec \dots \prec F_k$. Suppose additionally that $x_i+y_i$ is constant for all $i$. 
Let $\fP^\circ$ is the set of non-intersecting tuples $\bP=(P_1,P_2,\dots,P_k)$ such that $P_i\in\P_{I_i\to F_i}$ for $1\le i\le k$.
Then 
$$\sum_{\bP\in\fP^\circ} q^{\maj(\bP)}=\det_{1\le i,j\le k}\left(q^{i(x_i-x_j)}\qbin{u_i-x_j+v_i-y_j}{u_i-x_j}\right).$$
\end{theorem}

Let us introduce some notation for the proof. For a path $P\in\P_{I\to F}$, where $I,F\in\Z^2$, and a two-dimensional vector $\uu$ with integer coordinates, define $P+\uu\in\P_{I+\uu\to F+\uu}$ to be the path obtained by translating $P$ by $\uu$. Consider the vectors $\ee=(1,0)$ and $\nn=(0,1)$, and define the bijection 
\begin{equation}\label{eq:T}
\begin{array}{rrcl} T:&\paths{I}{F}{I'}{F'}&\to&\paths{I+\ee}{F+\ee}{I'+\nn}{F'+\nn}\\
&(P,Q)&\mapsto&(P+\ee,Q+\nn).
\end{array}
\end{equation}

\begin{proof}[Proof of Theorem~\ref{thm:non-int}]
As in~\cite{Krat-nonint}, let $\S_k$ denote the symmetric group, and consider the larger set 
$$\fP=\bigcup_{\sigma\in\S_k}\{(P_1,P_2,\dots,P_k): P_i\in\P_{I_{\sigma(i)}\to F_i} \text{ for all }i\}$$
of all tuples, whether intersecting or not. To each $\bP\in\fP$, assign a weight
$$w(\bP)=\sgn(\sigma) \ q^{\sum_{i=1}^k i(x_i-x_{\sigma(i)})} q^{\maj(\bP)},$$
and note that $w(\bP)=q^{\maj(P)}$ if $\bP\in\fP^\circ$, since in this case $\sigma$ must be the identity. The weighted sum of all tuples, using Lemma~\ref{lem:qbin}, is 
\begin{align*}
\sum_{\bP\in\fP} w(\bP)&=\sum_{\sigma\in\S_k} \sgn(\sigma) \prod_{i=1}^k q^{i(x_i-x_{\sigma(i)})} \sum_{P_i\in\P_{I_{\sigma(i)}\to F_i}} q^{\maj(P_i)}\\
&=\sum_{\sigma\in\S_k} \sgn(\sigma) \prod_{i=1}^k q^{i(x_i-x_{\sigma(i)})} \qbin{u_i-x_{\sigma(i)}+v_i-y_{\sigma(i)}}{u_i-x_{\sigma(i)}}\\
&=\det_{1\le i,j\le k}\left(q^{i(x_i-x_j)}\qbin{u_i-x_j+v_i-y_j}{u_i-x_j}\right),
\end{align*}
which equals the determinant in the statement.

Thus, it suffices to show that the contributions of all intersecting tuples to the weighted sum cancel out, leaving only
$\sum_{\bP\in\fP^\circ} w(\bP)=\sum_{\bP\in\fP^\circ} q^{\maj(\bP)}$. This key step is achieved by constructing an involution $\Phi$ on the set $\fP^{\Join}=\fP\setminus\fP^\circ$ of intersecting tuples, having the property that $w(\Phi(\bP))=-w(\bP)$. 
Such an involution is given in~\cite{Krat-nonint}, based on a four-step bijection described in \cite[Prop.\ 27]{Krat-nonint} in terms of two-rowed arrays. Instead, here we present an involution $\Phi$ that relies on the map $\phiC$ from Definition~\ref{def:phiC}, and so it has a simple visualization in terms of paths. 

Let $\bP=(P_1,P_2,\dots,P_k)\in\fP^{\Join}$ be an intersecting tuple, where $P_i\in\P_{I_{\sigma(i)}\to F_i}$ for $1\le i\le k$, for some 
$\sigma\in\S_k$.
As in~\cite{Krat-nonint}, of all intersection points between neighboring paths (i.e., $P_i$ and $P_{i+1}$ for some $i$), consider the ones with largest $x$-coordinate and, among them, let $D$ be the one with largest $y$-coordinate. Let $j$ be the smallest index such that $P_j$ and $P_{j+1}$ intersect at point $D$.
Let 
\begin{equation}\label{eq:TphiT} (P'_j,P'_{j+1})=T^{-1}\phiC T(P_j,P_{j+1}),\end{equation}
where $C=D+\ee+\nn$ and $T$ is defined in~\eqref{eq:T}, and let
$$\Phi(\bP)=(P_1,\dots,P_{j-1},P'_j,P'_{j+1},P_{j+2},\dots,P_k).$$
See Figure~\ref{fig:Phi} for an example.

\begin{figure}[htb]
\centering
\begin{tikzpicture}[scale=0.4]
 \draw[gray,dotted,very thick,fill](0,0) \start\N\E\N\N\N\N\E\N;
    \draw[gray] (0.2,4) node {$P_1$};
 \draw[red,very thick,fill](-1,1) \start\E\N\E\E\N\E\N\N;
    \draw[red] (-1,1.9) node {$P_2$};
    \draw[blue,dashed,very thick,fill](1,-1) \start\N\N\N\E\E\E\N;
    \draw[blue] (1.9,-.5) node {$P_3$};
\intersection{2}{2};
	\draw (2,2) node[below] {$D$}; 
	\draw (-1,1) node[left] {$I_1$};
	\draw (0,0) node[left] {$I_2$};
	\draw (1,-1) node[left]  {$I_3$};
	\draw (2,6) node[right] {$F_1$};
	\draw (3,5) node[right] {$F_2$};
	\draw (4,3) node[right] {$F_3$}; 
	\draw[<->] (6.8,1.5)-- node[above]{$\Phi$} (8.2,1.5);
	\draw[->] (1.5,-2)-- node[right]{$T$} (1.5,-3);

\begin{scope}[shift={(0,-8)}] 
   \draw[red,very thick,fill](0,1) \start\E\N\E\E\N\E\N\N;
    \draw[red] (0,1.8) node {$P$};
    \draw[blue,dashed,very thick,fill](1,0) \start\N\N\N\E\E\E\N;
    \draw[blue] (1.8,0.5) node {$Q$};
	\draw (0,1) node[left] {$I_1+\ee$};
	\draw (1,0) node[left]  {$I_3+\nn$};     
	\draw (4,5) node[right] {$F_2+\ee$};
	\draw (4,4) node[right] {$F_3+\nn$};    
\crossing{3}{3};
	\draw (3,3) node[below right] {$C$}; 
	\draw[<->] (6.8,1.5)-- node[above]{$\phiC$} (8.2,1.5); 
\end{scope}

\begin{scope}[shift={(13,-8)}] 
    \draw[red,very thin,dotted,fill](0,1) \start\E\N\E\E\N;
    \draw[blue,very thin,dotted,fill](1,0) \start\N\N\N\E\E;
    \draw[magenta,very thick,fill](2,-1) \start\E\N\N\N\N\E\N\N;
    \draw[magenta] (3.8,-.5) node {$R$};
    \draw[cyan,dashed,very thick,fill](-1,2) \start\E\E\N\E\E\E\N;
    \draw[cyan] (-.5,2.8) node {$S$};
	\draw (-1,2) node[left] {$I_1+\nn$};
	\draw (2,-1) node[left]{$I_3+\ee$};     
	\draw (4,5) node[right] {$F_2+\ee$};
	\draw (4,4) node[right] {$F_3+\nn$};    
\crossing{3}{3};
	\draw (3,3) node[below right] {$C$}; 
\end{scope}

\begin{scope}[shift={(13,0)}] 
 \draw[gray,dotted,very thick,fill](0,0) \start\N\E\N\N\N\N\E\N;
    \draw[gray] (0.2,4) node {$P_1$};
    \draw[magenta,very thick,fill](1,-1) \start\E\N\N\N\N\E\N\N;
    \draw[magenta] (3,-.5) node {$P'_2$};
    \draw[cyan,dashed,very thick,fill](-1,1) \start\E\E\N\E\E\E\N;
    \draw[cyan] (-.5,1.8) node {$P'_3$};
    \intersection{2}{2};
	\draw (2,2) node[below right] {$D$}; 
	\draw (-1,1) node[left] {$I_1$};
	\draw (0,0) node[left] {$I_2$};
	\draw (1,-1) node[left]  {$I_3$};
	\draw (2,6) node[right] {$F_1$};
	\draw (3,5) node[right] {$F_2$};
	\draw (4,3) node[right] {$F_3$};
	\draw[->] (1.5,-2)-- node[right]{$T$} (1.5,-3);
\end{scope}
\end{tikzpicture}

\caption{An example of the involution $\Phi$. The tuple $\bP=(P_1,P_2,P_3)$ in the top left has $\sigma=213$, and the point $D$ is an intersection of $P_2$ and $P_3$. Applying $T$ to this pair gives the pair $(P,Q)$ in the bottom left, to which we apply $\phiC$, followed by $T^{-1}$, to obtain the pair $(P'_2,P'_3)$ in the top right.}
\label{fig:Phi}
\end{figure}
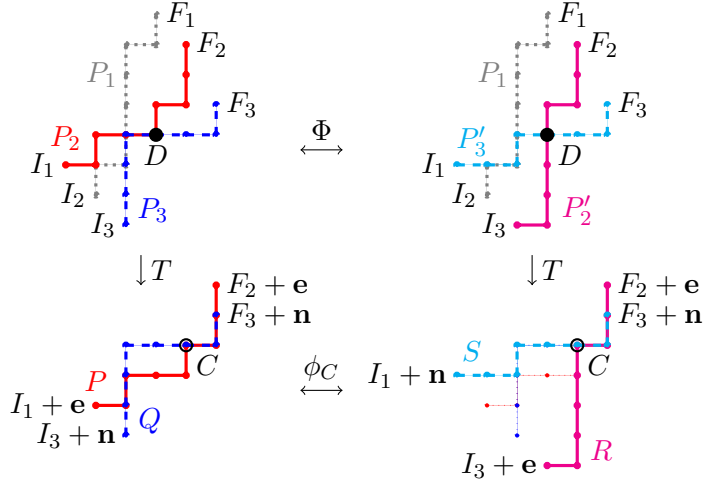

Let us show that $\Phi$ is well defined and that it is an involution on $\fP^{\Join}$. 
Let $(P,Q)=T(P_j,P_{j+1})\in\paths{I_{\sigma(j)}+\ee}{F_j+\ee}{I_{\sigma(j+1)}+\nn}{F_{j+1}+\nn}$. 
Since $D$ is the last intersection point of $P_j$ and $P_{j+1}$, and $F_j\prec F_{j+1}$, the steps of $P_j$ and $P_{j+1}$ that start at $D$ must be an $N$ and an $E$, respectively. The endpoints of these steps, namely $D+\nn$ in $P_j$ and $D+\ee$ in $P_{j+1}$, become a common point $C=D+\ee+\nn$ of the translated paths $P$ and $Q$. Note that the step of $P$ ending at $C$ is an $N$ and the step of $Q$ ending at $C$ is an $E$, i.e., $(P,Q)\in\pathsNC{I_{\sigma(j)}+\ee}{F_j+\ee}{I_{\sigma(j+1)}+\nn}{F_{j+1}+\nn}$.

Applying the map 
$$\phiC:\pathsNC{I_{\sigma(j)}+\ee}{F_j+\ee}{I_{\sigma(j+1)}+\nn}{F_{j+1}+\nn}\to \pathsNC{I_{\sigma(j+1)}+\ee}{F_j+\ee}{I_{\sigma(j)}+\nn}{F_{j+1}+\nn}$$
from Definition~\ref{def:phiC} and Lemma~\ref{lem:phiC}, we obtain a pair $(R,S)=\phiC(P,Q)$. Here $I_{\sigma(j)}+\ee$ plays the role of $A_1$ in the definition, and $A_1-\vv=I_{\sigma(j)}+\ee-\vv=I_{\sigma(j)}+\nn$. Similarly, $I_{\sigma(j+1)}+\nn$ plays the role of $A_2$, and $A_2+\vv=I_{\sigma(j+1)}+\nn+\vv=I_{\sigma(j+1)}+\ee$.

Finally, $(P'_j,P'_{j+1})=T^{-1}(R,S)=(R-\ee,S-\nn)\in \paths{I_{\sigma(j+1)}}{F_j}{I_{\sigma(j)}}{F_{j+1}}$.
Since the step of $R$ ending at $C$ is an $N$ and the step of $S$ ending at $C$ is an $E$, the translated paths $P'_j$ and $P'_{j+1}$ intersect at the point $D=C-\ee-\nn$. It follows that $\Phi(\bP)\in\fP^{\Join}$.
Additionally, since the map $\phiC$ does not change the steps ending at $C$ nor all the subsequent steps of either path, all the steps lying north or east of $D$ remain unchanged in $\Phi(\bP)$.
Thus, to compute the image by $\Phi$ of the tuple $\Phi(\bP)$, one would apply $\phiC$ to the pair $T(P'_j,P'_{j+1})=(R,S)$, recovering $(P,Q)$ (since $\phiC$ is an involution by Lemma~\ref{lem:phiC}), and then apply $T^{-1}$ to this pair, obtaining 
$T^{-1}\phiC T(P'_j,P'_{j+1})=(P_j,P_{j+1})$, so that $\Phi(\Phi(\bP))=\bP$. This proves that $\Phi$ is an involution on $\fP^{\Join}$. 

It remains to show that $\Phi$ is sign-reversing. Let  
$\sigma'\in\S_k$ be the permutation with $\sigma'(j)=\sigma(j+1)$, $\sigma'(j+1)=\sigma(j)$, and $\sigma'(i)=\sigma(i)$ for $i\notin\{j,j+1\}$, so that the $i$th component of $\Phi(\bP)$ is a path in $\P_{I_{\sigma'(i)}\to F_i}$ for $1\le i\le k$. 
By Lemma~\ref{lem:phiC}, and noting that the $x$-coordinates of the initial points of $P$ and $Q$ are $x_{\sigma(j)}+1$ and $x_{\sigma(j+1)}$, respectively, we have
\begin{align*}\maj(P'_j)+\maj(P'_{j+1})&=\maj(R)+\maj(S)\\
&=\maj(P)+\maj(Q)-(x_{\sigma(j+1)}-(x_{\sigma(j)}+1)+1)\\
&=\maj(P_j)+\maj(P_{j+1})-(x_{\sigma(j+1)}-x_{\sigma(j)}).
\end{align*}
We conclude that
\begin{align*}w(\Phi(\bP))&=\sgn(\sigma') \ q^{\sum_{i=1}^k i(x_i-x_{\sigma'(i)})} q^{\maj(\Phi(\bP))}\\
&=-\sgn(\sigma) \ q^{\sum_{i=1}^k i(x_i-x_{\sigma(i)})+x_{\sigma(j+1)}-x_{\sigma(j)}} q^{\maj(\bP)-(x_{\sigma(j+1)}-x_{\sigma(j)})}\\
&=-w(\bP).\qedhere
\end{align*}
\end{proof}

\subsection{A different bijective proof of Theorem~\ref{thm:pairs}}\label{sec:noq}

The standard proof of the Lindstr\"om--Gessel--Viennot formula, which is the $q=1$ specialization of Theorem~\ref{thm:non-int}, uses a simpler involution based on prefix-swapping.
Specifically, let $\pathsPD{A_1}{B_\de}{A_2}{B_\bde}$ be the subset of $\paths{A_1}{B_\de}{A_2}{B_\bde}$ consisting of intersecting pairs, and define an involution
$$\xi:\pathsPD{A_1}{B_\de}{A_2}{B_\bde}\to\pathsPD{A_2}{B_\de}{A_1}{B_\bde}$$
by cutting the paths at their last intersection point and simply swapping the resulting prefixes.
In the standard proof, the involution $\xi$ plays the role of $T^{-1}\phiC T$ in Equation~\eqref{eq:TphiT}. Note that $\xi$ does not behave well with respect to the total major index, which is why it was not used in the proof of Theorem~\ref{thm:non-int}.

In this subsection, we sketch how a similar prefix-swapping bijection can be used instead of $\phiC$ in order to prove Theorem~\ref{thm:pairs}, which enumerates pairs of paths by their number of crossings without tracking the major index.
Define a variation of $\xi$ by cutting the paths at their {\em first} intersection point, instead of the last, and then swapping the resulting prefixes; denote this involution by
$$\xi':\pathsPD{A_1}{B_\de}{A_2}{B_\bde}\to\pathsPD{A_2}{B_\de}{A_1}{B_\bde}.$$
Now let $(P,Q)\in\paths{A_1}{B_\de}{A_2}{B_\bde}$, and suppose that $P$ and $Q$ have a common point from where $P$ leaves with an $E$ step and $Q$ leaves with an $N$ step.
Define
$$\om(P,Q)= T \xi' T^{-1}(P,Q).$$
Note that the condition on $(P,Q)$ guarantees that the pair $T^{-1}(P,Q)=(P-\ee,Q-\nn)$ intersects, so $\om$ is well defined.

In the case that $A_1\prec A_2$ or $A_1=A_2$, one can also define $\om(P,Q)$ directly as follows. Consider the first common point of $P$ and $Q$ from where $P$ leaves with an $E$ step and $Q$ leaves with an $N$ step.
Call the vertices of $P$ and $Q$ immediately after this step the {\em cutting vertices}, and write $P=P_\triangleleft P_\triangleright$ and $Q=Q_\triangleleft  Q_\triangleright$ by splitting each path at its cutting vertex. Now swap the prefixes $P_\triangleleft$ and $Q_\triangleleft$ to obtain a pair $\om(P,Q)=(R,S)$ where
$$R= Q_\triangleleft  P_\triangleright \in\P_{A_2+\vv\to B_\de} \quad \text{and}\quad S=P_\triangleleft  Q_\triangleright\in\P_{A_1-\vv\to B_\bde}.$$
The condition $A_1\prec A_2$ or $A_1=A_2$ guarantees that the cutting vertices of $P$ and $Q$ correspond to the first intersection point of $P-\ee$ and $Q-\nn$. See Figure~\ref{fig:om} for an example. 

\begin{figure}[htb]
\centering
\begin{tikzpicture}[scale=0.4]
    \draw[red,very thick,fill](0,1) \start\E\N\E\E\N\E\N\N;
    \draw[red] (0,1.8) node {$P$};
    \draw[blue,dashed,very thick,fill](1,0) \start\N\N\N\E\E\E\N\E\E;
    \draw[blue] (1.8,0.5) node {$Q$};
	\draw (0,1) node[left] {$A_1$};
	\draw (1,0) node[left] {$A_2$};     
	\draw (4,5) node[right] {$B_\de$};
	\draw (6,4) node[right] {$B_\bde$};    
	\cutting{1}{3}
	\cutting{2}{2}
	\draw[->] (8.5,2)-- node[above]{$\om$} (9.5,2);
	\draw[<-] (2,-1.5)-- node[right]{$T$} (2,-2.5);

\begin{scope}[shift={(0,-8)}] 
    \draw[red,very thick,fill](-1,1) \start\E\N\E\E\N\E\N\N;
    \draw[blue,dashed,very thick,fill](1,-1) \start\N\N\N\E\E\E\N\E\E;
	\draw (-1,1) node[left] {$A_1-\ee$};
	\draw (1,-1) node[left] {$A_2-\nn$};     
	\draw (3,5) node[right] {$B_\de-\ee$};
	\draw (7,3) node[above] {$B_\bde-\nn$};   
	\cutting{1}{2}  
	\draw[<->] (8.5,1.5)--node[above]{$\xi'$} (9.5,1.5);
\end{scope}

\begin{scope}[shift={(15,-8)}] 
    \draw[purple,very thick,fill](1,-1) \start\N\N\N \E\N\E\N\N;
    \draw[violet,dashed,very thick,fill](-1,1) \start\E\N\E \E\E\E\N\E\E;
	\draw (-1,1) node[left] {$A_1-\ee$};
	\draw (1,-1) node[left] {$A_2-\nn$};     
	\draw (3,5) node[right] {$B_\de-\ee$};
	\draw (7,3) node[above] {$B_\bde-\nn$};   
	\cutting{1}{2}  
\end{scope}

\begin{scope}[shift={(15,0)}] 
   \draw[purple,very thick,fill](2,-1) \start\N\N\N \E\N\E\N\N;
    \draw[purple] (2.8,0.5) node {$R$};
    \draw[violet,dashed,very thick,fill](-1,2) \start\E\N\E \E\E\E\N\E\E;
    \draw[violet] (-1,3) node {$S$};
	\draw (-1,2) node[left] {$A_1-\vv$};
	\draw (2,-1) node[left] {$A_2+\vv$};
	\draw (4,5) node[right] {$B_\de$};
	\draw (6,4) node[right] {$B_\bde$};
	\cutting{1}{3}
	\cutting{2}{2}
	\draw[<-] (2,-1.5)-- node[right]{$T$} (2,-2.5);
	\draw[->] (7.5,2)-- node[above]{$\varsigma$} (8.5,2);
\end{scope}

\begin{scope}[shift={(28.5,0)}] 
   \draw[purple,very thick,fill](-1,2) \start\E\N\E \E\E\E\N\E\E;
    \draw[violet,dashed,very thick,fill](2,-1) \start\N\N\N \E\N\E\N\N;
	\draw (-1,2) node[left] {$A_1-\vv$};
	\draw (2,-1) node[left] {$A_2+\vv$};
	\draw (4,5) node[right] {$B_\de$};
	\draw (6,4) node[right] {$B_\bde$};
	\cutting{1}{3}
	\cutting{2}{2}
\end{scope}

\end{tikzpicture}
\caption{The bijection $\om$ and its relationship with the prefix-swapping involution $\xi'$. The crosses indicate the cutting vertices in $P$ and $Q$, which correspond to the first intersection point of the translated paths.}
\label{fig:om}
\end{figure}
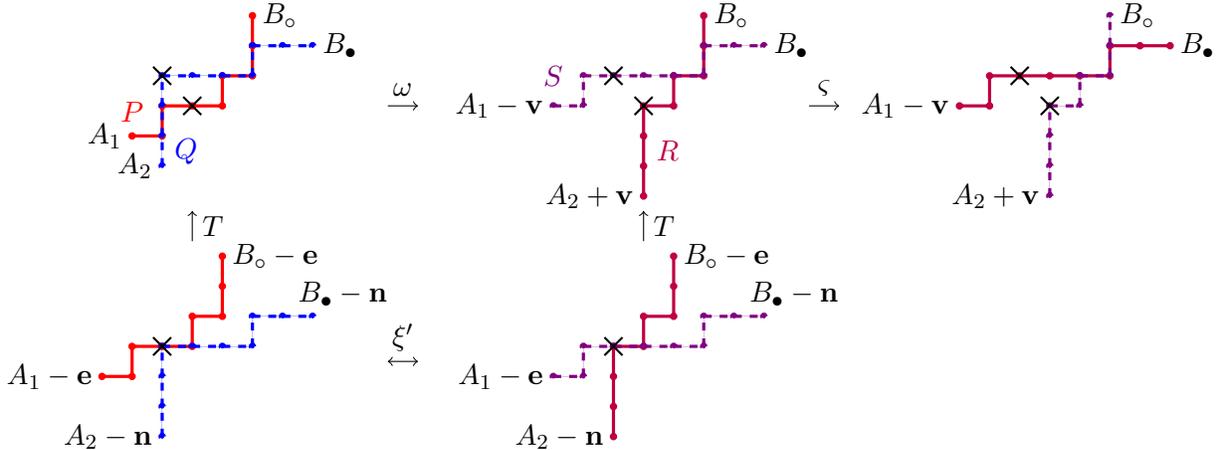

Let $\bom=\varsigma\om$, where $\varsigma$ is the swap from Equation~\eqref{eq:swap}. Note that $\bom(P,Q)\in\paths{A_1-\vv}{B_\bde}{A_2+\vv}{B_\de}$. 

\begin{lemma}\label{lem:om}
Suppose that $A_1\prec A_2$. Let $r\ge1$ if $B_\de=B_\bde$, and let $r\ge2$ otherwise.
Then the map $\bom$ defined above is a bijection
$$\bom:\pathsP{A_1}{B_\de}{A_2}{B_\bde}{r}\to\pathsP{A_1-\vv}{B_\bde}{A_2+\vv}{B_\de}{r-1}.$$
\end{lemma}

\begin{proof}
Let $(P,Q)\in\pathsP{A_1}{B_\de}{A_2}{B_\bde}{r}$, and let $(S,R)=\bom(P,Q)$. 
Since $A_1\prec A_2$, the cutting vertices of $P$ and $Q$ precede all the crossings except for the first one, and so
$S$ and $R$ have the same crossings as $P$ and $Q$ minus the first one.
It follows that $(S,R)\in\pathsP{A_1-\vv}{B_\bde}{A_2+\vv}{B_\de}{r-1}$.

To show that $\bom$ is a bijection, let us describe its inverse. Given any pair $(S,R)\in\pathsP{A_1-\vv}{B_\bde}{A_2+\vv}{B_\de}{r-1}$, we can determine the cutting vertices in each path by finding the first intersection of $R-\ee$ and $S-\nn$ (see Figure~\ref{fig:om}). The fact that this intersection exists is clear if $r\ge2$, since $S$ and $R$ cross in this case,
and in the case $r=1$ it is implied by the conditions $B_\de=B_\bde$ and $A_1\prec A_2$.

Cutting $R$ and $S$ at these vertices and swapping their prefixes, so that the resulting paths start at $A_1$ and $A_2$, we recover the unique pair $(P,Q)$ such that 
$\bom(P,Q)=(S,R)$.
\end{proof}

The proof of Theorem~\ref{thm:pairs} that we gave in Section~\ref{sec:pairs-proofs} can now be modified as follows. In Case~1, the bijections $\Bij_{2m}$ and $\Bij_{2m-1}$ can be replaced with the following simpler bijections that repeatedly apply $\bom$:
\begin{align*}
\pathsP{A_1}{B_2}{A_2}{B_1}{2m+1}=\pathsP{A_1}{B_2}{A_2}{B_1}{2m}&\stackrel{\bom^{2m}}{\to}\pathsP{A_1-2m\vv}{B_2}{A_2+2m\vv}{B_1}{1}=\paths{A_1-2m\vv}{B_2}{A_2+2m\vv}{B_1},\\
\pathsP{A_1}{B_1}{A_2}{B_2}{2m}=\pathsP{A_1}{B_1}{A_2}{B_2}{2m-1}&\stackrel{\bom^{2m-1}}{\to}\pathsP{A_1-(2m-1)\vv}{B_2}{A_2+(2m-1)\vv}{B_1}{1}\\
&\qquad =\paths{A_1-(2m-1)\vv}{B_2}{A_2+(2m-1)\vv}{B_1},
\end{align*}
from where Equations~\eqref{eq:switched-noq} and~\eqref{eq:same-noq} follow, using Equation~\eqref{eq:m=0noq}. See Figure~\ref{fig:Om} for an example.

\begin{figure}[htb]
\centering
\begin{tikzpicture}[scale=0.4]
    \draw[red,very thick,fill](0,1) \start\N\N\E\E\E\N\N\N\N\E\E\N\E\N\E\E;
    \draw[red] (0,2.5) node[left] {$P$};
    \draw[blue,dashed,very thick,fill](1,0) \start\N\N\N\N\E\N\E\E\N\N\E\E\N\N\N\N\E;
    \draw[blue] (1,1.5) node[right] {$Q$};
	\cutting{1}{4}
	\cutting{2}{3}
	\draw (0,1) node[left] {\small $A_1$};
	\draw (1,0) node[left] {\small $A_2$};     
	\draw (7,11) node[right] {\small $B_1$};
	\draw (8,9) node[right] {\small $B_2$};     
	\draw[->] (8.5,5.5)-- node[above]{$\bom$} (9.5,5.5);

\begin{scope}[shift={(12,0)}] 
    \draw[purple,very thick,fill](-1,2) \start\N\N\E\E \E\N\E\E\N\N\E\E\N\N\N\N\E;
    \draw[violet,dashed,very thick,fill](2,-1) \start\N\N\N\N \E\N\N\N\N\E\E\N\E\N\E\E;
  	\cuttingthin{1}{4}
	\cuttingthin{2}{3}
	\cutting{3}{6}
	\cutting{4}{5}
	\draw (-1,2) node[left] {\small $A_1-\vv$};
	\draw (2,-1) node[left] {\small $A_2+\vv$};     
	\draw (7,11) node[right] {\small $B_1$};
	\draw (8,9) node[right] {\small $B_2$};     
	\draw[->] (8.5,5.5)-- node[above]{$\bom$} (9.5,5.5);
\end{scope}

\begin{scope}[shift={(26,0)}] 
    \draw[brown,very thick,fill](-2,3)  \start\N\N\E\E\E\N\E\E \N\E\E\N\E\N\E\E;
    \draw[olive,dashed,very thick,fill](3,-2) \start\N\N\N\N\E\N\N\N \N\N\E\E\N\N\N\N\E;
    	\cuttingthin{3}{6}
	\cuttingthin{4}{5}
	\draw (-2,3) node[left] {\small $A_1-2\vv$};
	\draw (3,-2) node[left] {\small $A_2+2\vv$};     
	\draw (7,11) node[right] {\small $B_1$};
	\draw (8,9) node[right] {\small $B_2$};     
\end{scope}
\end{tikzpicture}
\caption{The bijection $\om^2$ applied to the pair of paths from the left of Figure~\ref{fig:Bij}.}
\label{fig:Om}
\end{figure}
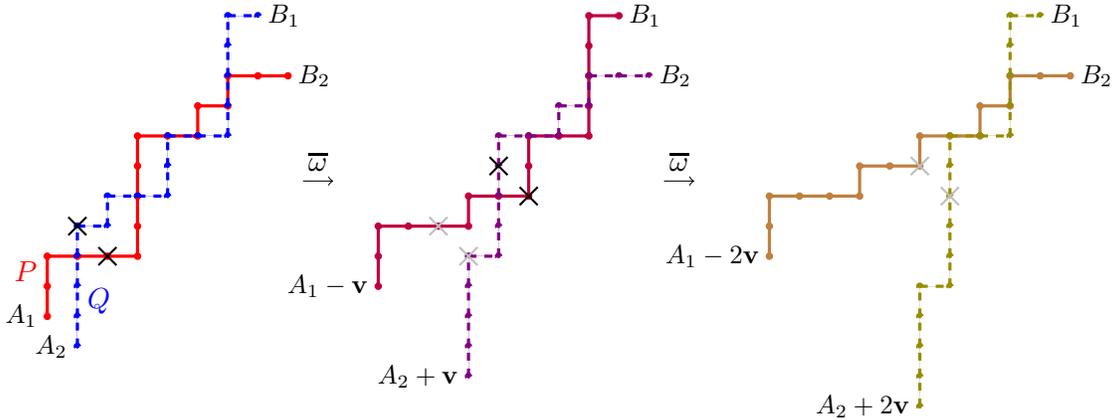

In Case~3, we can use the bijection
$$\bom^r:\pathsP{A_1}{B}{A_2}{B}{r}\to \pathsP{A_1-r\vv}{B}{A_2+r\vv}{B}{0}$$
to prove Equation~\eqref{eq:B1=B2-noq}. 
Equation~\eqref{eq:A1=A2-noq}, corresponding to Case~2, follows now by symmetry, rotating the paths by $180^\circ$. Note that rotation does not preserve the major index, so this argument would not allow us to combine Cases~2 and~3 in the proof of the refined version.

It is also possible to modify the proof of Case 4 using a variation of the map $\bom$. However, the resulting argument is not significantly simpler than our proof using the maps~$\bij_r$.

\section{Further research}\label{sec:further}

The enumeration of lattice paths by major index is intertwined with their enumeration by the number of valleys; equivalently, the number of peaks, the number of turns, or the number of descents, depending on terminology. In~\cite[Thm.\ 3.6.1]{Krat-turns}, Krattenthaler enumerates $k$-tuples of non-intersecting paths by the number of peaks, giving another refinement of the Lindstr\"om--Gessel--Viennot determinantal formula. And in~\cite{KM}, Krattenthaler and Mohanty give formulas counting lattice paths that lie between two given lines with respect to the major index and the number of peaks. 

In a follow-up paper~\cite{part2}, we will refine Theorems~\ref{thm:xaxis}, \ref{thm:line}, and~\ref{thm:pairs_refined} by adding a variable that keeps track of the number of valleys of the paths. Unfortunately, our bijective proofs above do not yield refinements by the number of valleys, since the effect of the bijections $\bt$, $\bs$ and $\si$ on this statistic is not the same for all paths. For example, if $P\in\P^N_{A\to B}$, the number of valleys of $\bs(P)$ and $P$ are equal unless $P$ starts with an $E$, in which case $\bs(P)$ has one fewer valley than $P$. Worse still, the number of valleys of $P$ and $\si(P)$ can differ by $0$, $1$ or $-1$ depending on how $P$ starts. 
To circumvent this challenge, a different approach will be taken in~\cite{part2}, by instead constructing bijections in terms of two-rowed arrays like those used by Krattenthaler and Mohanty~\cite{Krat-turns,Krat-nonint,KM}. While these bijections do not have a natural description in terms of paths, they are suitable to track the number of valleys, in addition to the major index.

Finally, an open problem which is unlikely to have a simple solution would be to generalize Theorem~\ref{thm:pairs} (or the refined Theorem~\ref{thm:pairs_refined}) from pairs of paths  to $k$-tuples of paths, for arbitrary $k$, enumerating them by the total number of crossings. By a simple translation of the paths, similar to the map $T$ from Section~\ref{sec:connections}, tuples of non-crossing paths are in bijection with tuples of non-intersecting paths, so the special case of zero crossings is solved by the Lindstr\"om--Gessel--Viennot determinant.

\subsection*{Acknowledgments}

The author is grateful to Sylvie Corteel and Carla Savage for illuminating conversations and for pointing out the results in \cite{SeoYee},
and to Christian Krattenthaler for useful ideas and further references.

\end{document}